\documentclass[11pt, reqno]{amsart}
\usepackage{amssymb,latexsym,amsmath,amsfonts, amsthm}
\usepackage{latexsym}
\usepackage{graphicx}
\usepackage[numbers,sort&compress]{natbib}
\usepackage{color}
\usepackage[mathscr]{eucal}
\usepackage{comment}

\usepackage[linkcolor=blue, citecolor=black]{hyperref}
\hypersetup{colorlinks=true, urlcolor=blue}
\usepackage[nameinlink]{cleveref}

\numberwithin{equation}{section}

\theoremstyle{definition}
\newtheorem{definition}{Definition}[section]
\newtheorem{example}[definition]{Example}

\theoremstyle{remark}
\newtheorem{remark}[definition]{Remark}
\theoremstyle{plain}
\newtheorem{theorem}[definition]{Theorem}
\newtheorem{lemma}[definition]{Lemma}
\newtheorem{proposition}[definition]{Proposition}

\newtheorem{result}[definition]{Result}

\newtheorem{corollary}[definition]{Corollary}

\newcommand{\eps}{\varepsilon}

\newcommand{\al}{\alpha}

\newcommand{\smoo}{\mathcal{C}}
\newcommand{\hol}{\mathcal{O}}
\newcommand{\poly}{\mathscr{P}}

\newcommand{\rl}{{\sf Re}}
\newcommand{\imag}{{\sf Im}}

\newcommand{\mapp}{\longmapsto}

\newcommand\ba[1]{\overline{#1}}

\newcommand{\cplx}{\mathbb{C}}

\newcommand{\Gr}{{\sf Gr}}

\begin{document}

	\title[Uniform Approximation on polynomial polyhedra]{Uniform approximation on certain polynomial polyhedra in  $\mathbb{C}^2$}
	\author{Sushil Gorai and Golam Mostafa Mondal}
	\address{Department of Mathematics and Statistics, Indian Institute of Science Education and Research Kolkata,
		Mohanpur -- 741 246}
	\email{sushil.gorai@iiserkol.ac.in}
	\address{Department of Mathematics, Indian Institute of Science, Bangalore-- 560012, India}
	\email{golammostafaa@gmail.com, golammondal@iisc.ac.in}
	\thanks{}
	\keywords{polynomial approximation; polynomial convexity, polynomial hull, pluriharmonic functions}
	\subjclass[2020]{Primary: 32E30, 32E20 Secondary: 46J10, 32A65}
	
	\begin{abstract}
		In this paper we extend the dichotomy given by Samuelsson and Wold that can be thought of as an analogue of the Wermer maximality theorem in $\mathbb{C}^2$ for certain polynomial polyhedra. We consider complex non-degenerate simply connected polynomial polyhedra of the form $\Omega:=\{z\in\mathbb{C}^2: |p_1(z)|<1, |p_2(z)|<1\}$ such that $\overline{\Omega}$ is compact. Under a mild condition of the polynomials $p_1$ and $p_2$, we prove that either the uniform algebra, generated by polynomials and some continuous functions $f_1,\dots, f_N$ on the distinguished boundary that extends as pluriharmonic functions on $\Omega$, is all continuous functions on the distinguished boundary or there exists an algebraic variety in $\overline{\Omega}$ on which each $f_j$ is holomorphic. We also compute the polynomial hull of the graph of pluriharmonic functions in some cases where the pluriharmonic functions are conjugates of holomorphic polynomials. We also give a couple of general theorem about uniform approximation on the domains with low boundary regularity. 
	\end{abstract}

	\maketitle
	\begin{center}		
		\date{}
	\end{center}
	

	\section{Introduction and statements of the results}\label{S:intro}
	
	Let $\mathcal{C}(K)$ represent the collection of all continuous complex-valued functions defined on a compact set $K$ in $\mathbb{C}^n$ with the norm $\|f\|_{K} = \sup_{x \in K} |f(x)|.$ For given complex-valued continuous functions $f_1, \ldots, f_N$ on $K$, we use the notation $[z_1,z_2,\cdots,z_n,f_1, \ldots, f_N;K]$ to represent the norm-closed subalgebra of $\smoo(K)$ that they generate. In $\cplx$, Wermer's maximality theorem gives a very interesting dichotomy: {\em A uniform algebra on the unit circle generated by holomorphic polynomials and some continuous functions $f_1,\dots, f_N$ on $\mathbb{T}$ is either same as the uniform algebra generated by holomorphic polynomials restricted to the unit circle or the class of all continuous functions on the unit circle.} That says either the algebra is of all continuous function or each of the function $f_j$ is holomorphic. Note that, by solving Dirichlet problem, any continuous function on the unit circle is boundary value of a harmonic function on the unit disc. In higher dimension, any continuous function on the Shilov boundary does not always extends to pluriharmonic functions. 
	Pluriharmonic functions have holomorphic conjugates. Therefore, it is natural to consider functions that are boundary values of pluriharmonic functions to extend Wermer’s maximality theorem to higher dimensions. The study of uniform algebras generated by holomorphic and pluriharmonic functions in higher dimensions has been conducted by \v{C}irka \cite{Cirka69}, Izzo \cite{AIJAMS93,AIJ95}, Samuelsson and Wold \cite{SaW12}, and Izzo, Samuelsson, and Wold \cite{ISW16}.
	In a seminal paper \cite{SaW12}, Samuelsson and Wold proved the following result which extends Wermer's maximality theorem for the case of bidisc in $\cplx^2$.

	\begin{result}[Samuelsson-Wold]\label{R:SW_T2}
		Let $f_{j}\in \smoo(\mathbb{T}^2)$ for $j=1,\cdots,N, N\ge 1$ and assume that each $f_{j}$ extends to a pluriharmonic function on $\mathbb{D}^{2}.$ Then either $[z_1,z_2,f_1,\cdots,f_{N};\mathbb{T}^2]=\smoo(\mathbb{T}^2),$ or there exists a non-trivial algebraic variety $Z\subset \cplx^2$ with $Z\cap b\mathbb{D}^{2}\subset \mathbb{T}^2,$ and the pluriharmonic extensions of the $f_{j}$'s are holomorphic on $Z.$	
	\end{result}
	\Cref{R:SW_T2} uses the properties of the bidisc crucially. One can ask the question: {\em Does the same dichotomy hold when one considers other domains?} This does not have immediate answer. For some images of the bidisc under proper polynomial maps, symmetrized bidisc being one of them, the question is answered positively in \cite{SGGM23}.
	
	In this paper, we prove a similar result for certain polynomial polyhedra. Let $p_1, p_2, \cdots, p_{l}$ be holomorphic polynomials in $z_1, z_2, \ldots, z_n$. Consider the set $\mathfrak{D}_{l}$ known as the \textit{polynomial polyhedron}  defined by
	\begin{align*}
		\mathfrak{D}_{l} := \{z \in \mathbb{C}^n : |p_{1}(z)| < 1, \ldots, |p_{l}(z)| < 1\},
	\end{align*}
	where $l \geq n$. Define $\Sigma_{j} := \{z \in \mathbb{C}^n : |p_{j}(z)| = 1\}$ and $\Sigma_{J_k} := \Sigma_{j_1} \cap \ldots \cap \Sigma_{j_k}$, where $J_k = \{j_1, \ldots, j_k\}$ with $1 \leq j_1 < j_2 < \ldots < j_k \leq l$. Clearly, the topological boundary $b\mathfrak{D}_{l}$ of $\mathfrak{D}_{l}$ is contained in $ \bigcup_{j}\Sigma_{j}$, and following the definition from \cite{CDHV85}, we say $\mathfrak{D}_{l}$ is \textit{complex non-degenerate} if
	\begin{align*}
		dp_{j_1} \wedge \ldots \wedge dp_{j_k}(z) \neq 0 \text{ on } \Sigma_{J_k}, \quad 1 \leq k \leq n.
	\end{align*}
	
	\noindent If $\mathfrak{D}_{l}$ is a complex non-degenerate analytic polyhedron, then the Shilov boundary coincides with the \textit{distinguished boundary} $\Gamma_{\mathfrak{D}_{l}}$ of $\mathfrak{D}_{l}$ and is equal to $\Gamma_{\mathfrak{D}_{l}} = \bigcup' \Sigma_{J_n}$, where $\bigcup'$ represents the union over all monotone multi-indices $(j_1, \cdots, j_n)$. In particular, when the number of polynomials is $n$, then $\Gamma_{\mathfrak{D}_{n}} = \{z \in \mathbb{C}^{n} : |p_{j}(z)| = 1,~j = 1, \ldots, n\}$ (see \cite{CDHV85}, \cite{Hof60}).
	For function algebra on analytic polyhedron and Shilov boundary for uniform algebras on analytic polyhedron see \cite{CDHV85} (see also \cite{Hof60, HofRos62}), and for details on uniform algebras on arbitrary domain and Shilov boundary see \cite{Gam69}.
	
	\medskip

	\begin{theorem}\label{T:PPolyhrn_DistingBdry}
		Let $\mathfrak{D}_{2}$ be complex non-degenerate simply connected polynomial polyhedron given by $\overline{\mathfrak{D}}_{2}:=\{z\in\cplx^2: |p_1(z)|\le 1,|p_{2}(z)|\le 1\}.$ Assume that each leaf $\{z\in\cplx^2:\; p_1(z)=c\},$ $\{z\in\cplx^2:\;p_2(z)=c\}$ is simply connected for $|c|=1.$ Let $f_{j}\in \smoo(\Gamma_{\mathfrak{D}_{2}})$ for $j=1,\cdots,N, N\ge 1,$ and assume that each $f_{j}$ extends to a pluriharmonic function on $\mathfrak{D}_{2}.$ Then either
		\[
		[z_1,z_2,f_1,\cdots,f_{N};\Gamma_{\mathfrak{D}_{2}}]=\smoo(\Gamma_{\mathfrak{D}_{2}})
		\]
		\noindent or there exists a non-trivial algebraic variety $V\subset \cplx^2$ with $V\cap b\mathfrak{D}_{2}\subset \Gamma_{\mathfrak{D}_{2}},$ and the pluriharmonic extension of each of the functions $f_{j}$ is holomorphic on $V.$
	\end{theorem}

	\begin{remark}
		If we take $p_{1}(z)=z_1$ and $p_2(z)=z_2$ we get back Samuelsson-Wold theorem (\Cref{R:SW_T2}).
	\end{remark}
	\begin{remark}
		Since each leaf is a simply connected Riemann surface, therefore, the number of components of each of the sets $\{z\in\cplx^2:\;p_1(z)=c\}$ and $\{z\in\cplx:\;p_2(z)=c\}$ is one for all $|c|=1.$
	\end{remark}
	
	\smallskip

	We also provide a description of the polynomial hull of certain graphs over the distinguished boundary of some polynomial polyhedra.
	In general, finding the polynomial hull is extremely difficult. Here we consider a polynomial polyhedron $\mathfrak{D}_2=\{z\in\cplx^2:\; |p_1(z)|\leq 1, |p_2(z)|\leq 1\}$ such that $\Psi:\cplx^2\to\cplx^2$ defined by 
	\[
	\Psi(z):=(p_1(z), p_2(z))
	\]
	is a proper map. Let $P$ be a holomorphic polynomial. We will compute the polynomially convex hull of the graph of $\ba{P\circ\Psi}$ over the distinguished boundary of the polynomial polyhedron $\mathfrak{D}_2$ (see \Cref{sec-polyhull} for details). Jimbo \cite{Jimbo03,Jimbo05} gave the description of polynomial hulls of graphs of anti-holomorphic polynomials on the torus. Here, we treat the case when $f$ is the restriction of an anti-holomorphic polynomial in $\cplx^2$ to the distinguished boundary $\Gamma_{{\mathfrak{D}_{2}}}$ of a complex non-degenerate polynomial polyhedron $\mathfrak{D}_{2}$ in $\cplx^2.$ We present explicit descriptions of the hull of the graph of $f$ on $\Gamma_{{\mathfrak{D}_{2}}}.$ In this setting, we show that the set $\widehat{\Gr_{f}(\Gamma_{{\mathfrak{D}_{2}}})}\setminus \Gr_{f}(\Gamma_{{\mathfrak{D}_{2}}})$ is contained in an algebraic variety.
	
	\smallskip
	
	We now look at the graph over certain bounded domains whose closure is polynomially convex.  Specifically, we look at domains with polynomially convex closure and having segment property.
	
	\begin{definition}
		Let $\Omega$ be a bounded domain in $\cplx^n.$ $\Omega$ has \textit{segment property} if for every $\xi\in b\Omega,$ there exists a neighborhood $N_{\xi}$ of $\xi$ and a vector $V_{\xi}$ in $\cplx^n$ such that 
		\begin{align*}
			z+tV_{\xi}\in\Omega, \forall z\in\overline{\Omega}\cap N_{\xi}, 0<t<1.
		\end{align*}	
	\end{definition}
	
	\noindent In \cite{AveHedPer16} authors showed that a domain $\Omega$ has the segment property if, and only if, $\Omega$ has $\smoo^{0}$ boundary. For more details on the segment property we refer the reader to \cite{DeZo01}. We now present a couple of observations that will generalize results from (\cite[Theorem 1.3]{SaW12}) and  \cite[Theorem 1.1]{ISW16} respectively:

	\begin{theorem}\label{T:Non_smooth}
		Let
		$\Omega\subset \cplx^{n}$ be a bounded domain with $\widehat{\overline{{\Omega}}}=\overline{{\Omega}}$ and $\Omega$ possesses segment property except possibly at a countable set of boundary points. Let $h_{j}$'s be pluriharmonic on $\Omega$ and continuous on $\overline{\Omega}$ for $j=1,\cdots,N,$ and assume that $[z_1,\cdots,z_{n},h_1,\cdots,h_N;\partial\Omega]=\smoo(\partial\Omega).$ Then either there exists an analytic disc in $\Omega$ where all $h_{j}$'s are holomorphic or $[z_1,\cdots,z_{n},h_1,\cdots,h_N;\overline{\Omega}]=\smoo(\overline{\Omega}).$ 
	\end{theorem}

	\begin{theorem}\label{T:WMT_NonSmooth}
		Let
		$\Omega\subset \cplx^{n}$ be a bounded domain with $\widehat{\overline{{\Omega}}}=\overline{{\Omega}}$ and $\Omega$ possesses segment property except possibly at a countable set of boundary points. Let $h_{j}$'s be pluriharmonic on $\Omega$ and continuous on $\overline{\Omega}$ for $j=1,\cdots,N.$ Then the following are equivalent:
		\begin{enumerate}
			\item $G_h(\partial\Omega)$ is polynomially convex.
			\item $\widehat{G_h(\partial\Omega)}\setminus G_h(\partial\Omega)$ contains no analytic disc.
			\item There does not exist a nontrivial analytic disc $\triangle\to \Omega$ on which all the $h_{j}$'s are holomorphic.
			\item $[z,h;\overline{\Omega}]=\left\{f\in \smoo(\overline{\Omega}):f|_{\partial\Omega}\in [z,h;\partial\Omega]\right\}.$
		\end{enumerate}
	\end{theorem}

	\noindent We observe that, in (\cite[Theorem 1.1]{ISW16}) and (\cite[Theorem 1.3]{SaW12}), the $\smoo^1$-smoothness of $\Omega$ is required to prove  and  only to show that $\Gr_{h}(\overline{\Omega})$ is polynomially convex. By \Cref{L:Poly_Cnvx}, we can say that $\Gr_{h}(\overline{\Omega})$ is polynomially convex for any bounded domain $\Omega$ with polynomially convex closure and having segment property except possibly at a countable set of boundary points. The rest follows from \cite[Theorem 1.3]{SaW12} and \cite[Theorem 1.1]{ISW16}, respectively.
\smallskip

 \noindent Although the proof of \Cref{T:Non_smooth} and \Cref{T:WMT_NonSmooth} are not very different from that of \cite{ISW16} except for the polynomial convexity, but there are huge implication of \Cref{T:Non_smooth} and \Cref{T:WMT_NonSmooth} as we can deal with a huge class of nonsmooth domains here. A domain having continuous boundary satisfies the segment property. Although whether the closure of a bounded domain is polynomially convex is a very difficult problem, the class of bounded convex domains satisfies the segment property and their closure is also polynomially convex. Hence, we obtain the following corollary: 

 \begin{corollary}
     Let $\Omega$ be a bounded convex domain in $\cplx^n$. Let $h_j$, $j=1,\dots, N$, be pluriharmonic in $\Omega$ and continuous on $\overline{\Omega}$. Assume further that there does not exist a nontrivial analytic disc $\triangle\subset\Omega$ on which all the $h_{j}$'s are holomorphic.
     Then 
     \begin{enumerate}
			\item $G_h(\partial\Omega)$ is polynomially convex.
			\item $\widehat{G_h(\partial\Omega)}\setminus G_h(\partial\Omega)$ contains no analytic disc.
			\item $[z,h;\overline{\Omega}]=\left\{f\in \smoo(\overline{\Omega}):f|_{\partial\Omega}\in [z,h;\partial\Omega]\right\}.$
		\end{enumerate}
 \end{corollary}
	
	\smallskip
	
	We now provide some remarks: the first two of them concerns about the comparison between \Cref{T:PPolyhrn_DistingBdry} with known results on polynomial polyhedra and other domains in $\cplx^2$; and the other concerns about the possibility of extending \Cref{T:PPolyhrn_DistingBdry} to domains in $\cplx^n$, $n\geq 3$. These comments might be interesting in the context of further research.
	
	\begin{remark}
		The conditions on polynomial polyhedra in \Cref{T:PPolyhrn_DistingBdry} are quite natural. On simply connected domain any complex valued pluriharmonic function can be seen as a sum of a holomorphic function and the real part of another holomorphic function. This is used crucially in our proofs. Complex non-degeneracy of the polyhedra gives us the segment property which is one other essential ingredient in our proof. A natural class of polynomial polyhedra is special polynomial polyhedra in the sense of \cite{Nivoche}. The class of polyhedra that satisfies conditions of \Cref{T:PPolyhrn_DistingBdry} has nontrivial intersection with the class of special polynomial polyhedra. Any polynomially convex set satisfying some conditions (see Corollary~2 and Theorem~3 of \cite{Nivoche}) can be approximated by special polynomial polyhedra. It is not clear whether \Cref{T:PPolyhrn_DistingBdry} holds for special polynomial polyhedra defined in \cite{Nivoche}. Also, there are some approximation results of Mergelyan type on certain polynomial polyhedra (see \cite{Petrosyan2006}, \cite{Petrosyan2007}).
	\end{remark}
	
	\begin{remark}
		When one considers a general domain in $\cplx^2$, it is already difficult to find the Shilov boundary. Therefore, one assumes the boundary approximation as in \Cref{T:Non_smooth}. If Shilov boundary can be found precisely and it is two real dimension and the other parts of the boundary are foliated by analytic discs, then the question of getting a dichotomy as in \Cref{T:PPolyhrn_DistingBdry} make sense and is also interesting. For certain proper holomorphic images of the bidisc such dichotomy were shown in \cite{SGGM23}.
	\end{remark}
	
	\begin{remark}
		In complex dimension three or higher, even for the polydisc,
  things are complicated. Main obstruction comes in giving a stratification of the graph over the boundary. This is the reason for assuming boundary approximation while considering the general domains in \Cref{T:Non_smooth}.
		Here, unlike the case in \Cref{T:PPolyhrn_DistingBdry}, one need to deal with $n-1$ dimensional analytic varieties; some part of our techniques fail even in $\cplx^3$. The case of polynomial polyhedra could be thought of as the first step towards it. 
	\end{remark}

	We now give a brief outline for our proofs of \Cref{T:PPolyhrn_DistingBdry}.
	The philosophy behind the proof of \Cref{T:PPolyhrn_DistingBdry} closely follows that of \Cref{R:SW_T2} although adapting the relevant tools to our case requires some new insights. For example, in the proof of \Cref{T:PPolyhrn_DistingBdry}, the polynomial convexity of $h$ over the polynomial polyhedron is crucial, which we prove by showing that the polynomial polyhedron has segment property. We also gave a version of Tornehave's result for polynomial polyhedra. Firstly, we establish that \Cref{T:PPolyhrn_DistingBdry} holds if and only if the graph of $h$ over the distinguished boundary of the polynomial polyhedron $\mathfrak{D}_2$ is polynomially convex (refer to \Cref{T: Approx_Cont_Func}). Therefore, if the graph lacks polynomial convexity, we examine different scenarios through specific lemmas, demonstrating the existence of an analytic variety within $\mathfrak{D}_2$ where all $h_j$'s are holomorphic. Utilizing Tornehave's result (for polynomial polyhedron), we can find the required algebraic variety.
	\smallskip
	
	A few words about the layout of the paper: The first part of \Cref{S:technical} collects some known results that are required in this paper while in the second part we prove several new results which are crucial for this paper. In \Cref{S:proofMain} we present a proof of \Cref{T:PPolyhrn_DistingBdry}. The polynomial hulls of some graphs are computed in \Cref{sec-polyhull}. \Cref{sec-examples} is devoted to some examples.

	\section{Technical Results}\label{S:technical}
	\noindent We make use of the following result due to Samuelsson-Wold \cite[Corollary 4.2]{SaW12}.
	\begin{result}[Samuelsson-Wold]\label{R:Graph_Tot_Rl}
		Let $M$ be a complex manifold of dimension $n,$ let $h=(h_1,\cdots,h_{N}):M\to\cplx^{N}$ be a $\smoo^1$-smooth map. If $z\in M$ and there are functions $h_{j_1},\cdots,h_{j_n}$ such that $\overline{\partial}h_{j_1}\wedge\cdots\wedge \overline{\partial}h_{j_n}(z)\not=0,$ then the graph $\Gr_{h}(M)$ of $h$ at $(z,h(z))$ is totally real.
	\end{result}

	For $n>1,$ the most general result known for polynomial approximation is the following due to Samuelssonand Wold \cite{SaW12}:

	\begin{result}(Samuelsson-Wold)\label{R:Stratificatn_Apprx}
		Let $X$ be a polynomially convex compact set in $\cplx^{n}$ and assume that there are closed sets $X_{0}\subset X_{1}\subset \cdots\subset X_{N}=X$ such that $X_{j}\setminus X_{j-1},j=1,\cdots,N,$ is a totally real set. If $\smoo(X_0)=\hol(X_0),$ then $\smoo(X)=[z_1,\cdots,z_{n};X].$
	\end{result}

	The following result shows that plurisubharmonic functions can also be used to describe the polynomial convex hull of a compact set.
	
	\begin{result}[H\"{o}rmander, \cite{H}]\label{R:hormander}\index{Plurisubharmonic hull}
	\sloppy Let $G$ be a pseudoconvex open subset of $\cplx^n$ and $X$ be a compact subset $G.$ Then $\widehat{X}_{G} = \widehat{X}_{G}^P $, where $\widehat{X}_{G}^P:=\left\lbrace z \in G: u(z) \leq\sup\nolimits_{X} u \;\forall u \in {\sf psh}(G) \right\rbrace$ and $\widehat{X}_G:=
		\big\lbrace z\in G : |f(z)|\leq \sup\nolimits_{z\in X} |f(z)|\;\forall f\in \hol(G)\big\rbrace$.
	\end{result}
	\noindent In case $G= \cplx^n,$ \Cref{R:hormander} says that the plurisubharmonically convex hull of $X$ is same as the polynomially convex hull of $X.$

	\begin{result}[\cite{Rosay06,Ross60}]\label{R:Rossi}
		Let $E$ be a compact set in $\cplx^{n}.$ Let $a\in \widehat{E}\setminus E$, and let $N$ be a relatively compact neighborhood of $a$ that does not intersect $E$. Then, $a\in \widehat{{\widehat{E}\cap bN}}$ (where $bN$ denotes
		the topological boundary of $N$).
	\end{result}
	
	\begin{result}(\cite[Propostion 4.7]{SaW12})\label{R:Polycnvx_fiber}
		Let $K\subset \cplx^{N}$ be compact, let $F:\cplx^{n} \to\cplx^{m}$ be the uniform limit on $K$ of entire functions, and let $Y=F(K);$ note that $F$ extends to $\widehat{K}.$ If $y\in Y$ is a peak point for the algebra $\poly(Y),$ then $\widehat{K}\cap F^{-1}(y)=\widehat{(K\cap F^{-1}\{y\}}).$
	\end{result}


	Let ${\sf psh}(\Omega)$ denote the class of plurisubharmonic
	functions in $\Omega$ and ${\sf psh}(\overline{\Omega})$ denotes the collection of function $\phi$ such that there exists an open set $U\ \supset \Omega$ with $\phi \in  {\sf psh}(U)$. Sibony \cite[Theorem 2.2]{Sib87} showed that if $\Omega$ is a smooth pseudoconvex domain, then every functions in ${\sf psh}(\Omega)\cap C(\overline{\Omega})$ can be approximated uniformly with functions in ${\sf psh}(\overline{\Omega})\cap C^{\infty}(\overline{\Omega})$ and we say $\Omega$ has the \textit{PSH-Mergelyan property}. Forn\ae ss and Wiegerinck \cite[Theorem 1]{ForWie89}	proved that this is also true for any bounded $\smoo^1$-smooth domain.  Avelin, Hed, and Persson \cite{AHP16} further relaxed the boundary conditions to continuity. Subsequently, in \cite{HPJW17}, the authors presented an improved version of this result.
	
	\begin{result}\label{R:PshAprx}
		Let $\Omega$ be a bounded domain in $\cplx^{N}.$ Suppose that $\Omega$ has segment property except at a countable set of boundary points. Then $\Omega$ has the PSH-Mergelyan property.
	\end{result}

	The following definition is from \cite{KutSam2013}.
	\begin{definition}
		\emph{Two embeddings $\Phi,\Psi:\cplx\hookrightarrow\cplx^{n}$ are equivalent} if there exist automorphisms $\phi\in \text{Aut}(\cplx^n)$ and $\psi\in \text{Aut}(\cplx)$ such that $\phi\circ\Phi=\Psi\circ\psi.$
	\end{definition}
	
	\begin{result}\label{Rmk:Abhy_Sukuki}
		Every polynomial embedding of $\cplx$ into $\cplx^2$ is equivalent to the standard one, i.e. if a complex algebraic curve in $\cplx^2$ is isomorphic to $\cplx$ then in some
		algebraic coordinates $(z^{*}_1,z^{*}_{2})$ on $\cplx^{2}$ its equation is $z^{*}_{1}=0$  (see Abhyankar-Moh \cite{AbhMoh75}, and Suzuki \cite{Suzuki74}).
	\end{result}
	
	\par The following result \cite[Theorem 5.1]{WR82} gives a criterion for a homogeneous polynomial map from $\cplx^{n}$ to $\cplx^{n}$ to be proper holomorphic.
	\begin{result}[Rudin]\label{R:Hom_Proper}
		Suppose that
		\begin{enumerate}
			\item $d_1,\cdots,d_{n}$ are positive integers,
			\item $p_{j}:\cplx^{n}\to \cplx$ is a homogeneous polynomial of degree $d_{j},$ for $j=1,\cdots,n,$
			\item Define $F:\cplx^n\to \cplx^n $ by $F(z_1,\cdots,z_n)=(p_1(z_1,\cdots,z_n),\cdots,p_{n}(z_1,\cdots,z_n)).$
		\end{enumerate}
		Then $F:\cplx^n\to\cplx^n$ is proper if and only if $F^{-1}\{0\}=\{0\}.$
	\end{result}
	
	The following result is due to Izzo-Samuelsson-Wold \cite[Lemma 2.2]{ISW16}.
	
	\begin{result}[Izzo-Samuelsson-Wold]\label{R:Coplx_Pt_Variety}
		Let $\Omega\subset \cplx^n$ be a domain, let $h_{j}\in PH(\Omega)$ for $j = 1,\cdots,N,$ and
		$Z\subset \Omega$ be an irreducible analytic set of dimension $d\ge 1.$ Let $1\le m\le d,$ fix $(i_1,\cdots,i_m),$ and define	
		\begin{align*}
			Z':=\{z\in Z_{\text{reg}}:i^{*}_{Z_{\text{reg}}}(\overline{\partial}h_{j_1}\wedge\cdots\wedge \overline{\partial}h_{j_n})(z)=0\}.
		\end{align*}
		Then $\widetilde{Z}:=Z^{'}\cup Z_{sing}$ is an analytic subset of $\Omega.$
	\end{result}

	The following result is from \cite[Lemma 1.3]{Alexander1993}.
	
	\begin{result}[Alexander]\label{R:Component_Non_Empty}
		If $X$ is a compact subset of $\cplx^n$ and $E$ is a component of $\widehat{X}$, then
		$E=\widehat{X\cap E}$ (by a component of $X,$ we mean a maximal connected subset of $X$).
	\end{result}
	The following result is known as the Remmert proper mapping theorem. It will play a crucial role in our proofs. 
	
	\begin{result}[Remmert
		Proper Mapping theorem \cite{Rem56,Rem57}]\label{R:Remmert}
		Let $M, N$ be complex spaces and $f:M\to N$ is a proper holomorphic map. If $Z$ is an analytic subvariety in $M$ then $f(Z)$ is also an analytic subvariety in $N.$ Moreover, if $Z$ is irreducible then $f(Z)$ is also irreducible subvariety of $N.$
	\end{result}
	
	The following is known as Tornehave's result\index{Tornehave's result} (\cite[Corollary 3.8.11]{Sto07}).
	\begin{result}[Tornehave]\label{R:Tornehave_BiDisc}
		Let $V$ be a one-dimensional analytic subvariety of $\mathbb{D}^n$ such that $\overline{V}\setminus V\subset \mathbb{T}^n.$ Then there is an algebraic curve $Z$ in $\cplx^{n}$ with $V=Z\cap \mathbb{D}^n.$
	\end{result}
	The following is a very simple lemma.
	\begin{lemma}\label{L:Plurihrmonic_Cnjgte}
		Let $h$ be a complex valued pluriharmonic function on a simply connected domain $\Omega,$ then there exist holomorphic functions $\Psi$ and $\Phi$ on $\Omega$ such that $h=\Psi+\rl \Phi$ on $\Omega.$	
	\end{lemma}

	In the remaining part of this section we state and proof some results which are mostly new. These will be used in the proofs of our theorems.
	\begin{lemma}\label{L:Harmonic_bdry}
		Let $\mathfrak{D}_2=\{z\in \cplx^2: |p_{1}(z)|<1,|p_{2}(z)|<1\}$ be a bounded complex non-degenerate polynomial polyhedron and $f\in PH(\mathfrak{D}_2)\cap \smoo(\overline{\mathfrak{D}_2}).$ Then $f$ is harmonic on $b\mathfrak{D}_2.$	
	\end{lemma}
	\begin{proof}
		Let $\alpha\in \mathcal{A}:=\{|p_1(z)|=1,|p_{2}(z)|<1\}\subset b\mathfrak{D}_2.$ Since $\mathfrak{D}_2$ is complex non-degenerate, $\mathfrak{D}_2$ has segment property. Therefore, there exists a vector $V_{\alpha}$ and a neighborhood $N(\alpha)$ of $\alpha$ in $\cplx^{2}$ such that
		\begin{align*}
			z+tV_{\alpha}\in \mathfrak{D}_2, \forall z\in N(\alpha)\cap \overline{\mathfrak{D}_2} , \forall t\in (0,1).
		\end{align*}
		\noindent We choose an analytic disc $\phi:\mathbb{D}\to \cplx^2$ in $N(\alpha)\cap \mathcal{A}$ i.e., $\phi(\mathbb{D})\subset N(\alpha)\cap \mathcal{A}.$ Since harmonic is a local property, it is enough to show that $f\circ\phi:\mathbb{D}\to \cplx$ is harmonic. We define a sequence of analytic disc $\phi_{n}$ by 
		\begin{align*}
			\phi_{n}(\xi)=\phi(\xi)+\frac{1}{n}V_{\alpha}, \forall n\in \mathbb{N}.
		\end{align*}
		Since $\phi(\xi)\in (N(\alpha)\cap \mathcal{A}) \subset N(\alpha)\cap  \overline{\mathfrak{D}_2}, \forall \xi\in \mathbb{D},$ hence $\phi(\xi)+\frac{1}{n}V_{\alpha}\in \mathfrak{D}_{2}$ i.e., $\phi_{n}(\mathbb{D})\subset \mathfrak{D}_2 ~\forall n\in \mathbb{N}.$ Since $f\in PH(\mathfrak{D}_2)\cap \smoo(\overline{\mathfrak{D}_2}),$ $f\circ \phi_{n}:\mathbb{D}\to \cplx$ is harmonic for all $n.$ It is easy to see that $\phi_{n}$ converges to $\phi$ uniformly on $\mathbb{D}$. Indeed, 
		\begin{align*}
			\lim_{n\to \infty} \sup_{\mathbb{D}}\|\phi_{n}(\xi)-\phi(\xi)\|=\lim_{n\to\infty} \frac{1}{n}\|V_{\alpha}\|=0.
		\end{align*}
		Therefore, $f\circ\phi_{n}$ is converges to $f\circ\phi$ uniformly on $\mathbb{D}.$ Being uniform limit of sequence of harmonic functions, $f\circ \phi$ is also harmonic.
	\end{proof}
	
	\begin{remark}\label{Rmk:Pluriharminc_bdry}
		Let $\mathfrak{D}_{n}=\{z\in \cplx^n: |p_{1}(z)|<1,\cdots,|p_{n}(z)|<1\}$ be a bounded complex non-degenerate polynomial polyhedron and $f\in PH(\mathfrak{D}_n)\cap \smoo(\overline{\mathfrak{D}_n}).$ Then for every increasing multi-index $J_{k}$ of length $k\in\{1,\cdots,n-1\}$ the function $f$ is pluriharmonic on the $(n-k)-$dimensional complex manifold $\Sigma_{J_{k}}.$ 
	\end{remark}

	\par Next, we give a class of domains that has the segment property. Let $\psi_{1},\cdots,\psi_{N}:\cplx^{n}\to \mathbb{R}$ are $C^{1}$-smooth functions. Consider the bounded domain 
	\begin{align*}
		\Omega:=\{z\in \cplx^{n}:\psi_{1}<0,\cdots,\psi_{N}<0\}.
	\end{align*} 
	
	\noindent We denote $\Sigma_{j}:=\{z\in \cplx^n:\psi_{j}(z)=0\}$ and
	$\Sigma_{J_k}:=\Sigma_{j_1}\cap\cdots\cap\Sigma_{j_k}$, where $J_k=({j_1,\cdots, j_k})$ with  $1 \leq j_1 < j_2 < \ldots < j_k \leq N$. The topological boundary of $\Omega$ is denoted by  $\partial\Omega$ and clearly $\partial\Omega\subset\cup_{j}\Sigma_{j}.$
	
	\smallskip
	
	\noindent The result presented below describes a class of domains where the segment property holds. A special case for Weil polyhedra
	was proved in \cite[Proposition 2.1]{Petrosyan2006}. 
	\begin{proposition}\label{P:Segmnt_Prorty}
		Let $\Omega$ be as above. Assume that for any collection $\{j_1,\cdots j_k\}\subset\{1,2,\cdots,N\},$ $d\psi_{j_1}\wedge\cdots\wedge d\psi_{j_k}(z)\ne 0 \text{ on } \Sigma_{J_k}~~\forall~ 1\le k\le N.$
		Then $\Omega$ has the segment property.
	\end{proposition} 
	
	\begin{proof}
		Let $\xi\in \partial\Omega.$ This implies $\xi\in \Sigma_{J_k}$ for some $J_{k}$ i.e., there exist $j_1,\cdots,j_k$ such that 
		\begin{align*}
			\psi_{j}(\xi)&=0 ~~\forall j\in\{j_1,\cdots,j_{k}\} \text{ and }\\
			\psi_{s}(\xi)&<0 ~~\forall s\notin\{j_1,\cdots,j_{k}\}.
		\end{align*}
		
		\noindent Let us define $\Psi_{J_k}:=(\psi_{j_1},\cdots,\psi_{j_k}):\cplx^{n}\to\mathbb{R}^{k},$ where $J_k=({j_1,\cdots, j_k}).$ Then the differential $d\Psi_{J_{k}}$ of $\Psi_{J_{k}}$ at $\xi$ is given by
		
		\begin{align*}
			d\Psi_{J_{k}}|_{\xi}=\begin{pmatrix} 
				\frac{\partial \psi_{j_l}(\xi)}{\partial{z}_{m}} \vline\frac{\partial \psi_{j_l}(\xi)}{\partial{\bar{z}}_{m}}
			\end{pmatrix}^{m=1,\cdots,n}_{l=1,\cdots,k}
		\end{align*}
		
		\noindent	The condition $d\psi_{j_1}\wedge\cdots\wedge d\psi_{j_k}(z)\ne 0 \text{ on } \Sigma_{J_k}$ implies that $d\Psi_{J_{k}}$ has maximal rank on $\Sigma_{J_{k}}.$  Let $\widetilde{V}=(\widetilde{V_1},\cdots,\widetilde{V_k})\in \mathbb{R}^{k}$ such that $\widetilde{V_j}<0~~\forall j\in \{1,\cdots,k\}.$ Then there exists $W=(W_1,\cdots,W_{n})\in \cplx^{n}$ such that 
		\begin{align}\label{E:Negative Vectr}
			d\Psi_{J_{k}}|_{\xi}(W)=
			\begin{pmatrix} 
				\frac{\partial \psi_{j_l}(\xi)}{\partial{z}_{m}} \vline\frac{\partial \psi_{j_l}(\xi)}{\partial{\bar{z}}_{m}}
			\end{pmatrix}\begin{pmatrix}
				W \\
				\overline{W}
			\end{pmatrix}=
			\widetilde{V}.
		\end{align}
		
		\noindent If we take $w=W+\xi,$ where $w=(w_1,\cdots,w_n),$ then from (\ref{E:Negative Vectr}), we get that
		\begin{align*}
			\notag	\sum_{l=1}^{n}\frac{\partial \psi_{j}(\xi)}{\partial{{z_{l}}}}({w_l}-{\xi_l})+\sum_{l=1}^{n}\frac{\partial \psi_{j}(\xi)}{\partial\bar{{z_{l}}}}(\bar{w_l}-\bar{\xi_l})&<0~~\forall j\in\{j_1,\cdots,j_k\}.
		\end{align*}
		
		\noindent This implies
		\begin{align}\label{E:psi_zero}
			\rl\left<\partial\psi_{j}(\xi),\bar{w}-\bar{\xi}\right>&<0, \text{ where } \partial\psi_{j}(\xi)=\left(\frac{\partial \psi_{j}(\xi)}{\partial{z}_{1}} ,\cdots,\frac{\partial \psi_{j}(\xi)}{\partial{z}_{n}}\right).
		\end{align}	
		
		\noindent Next we claim that there exists a neighborhood $U(\xi)$ of $\xi$ such that 
		\begin{align*}
			\rl\left<\partial\psi_{j}(z),\bar{w}-\bar{\xi}\right>&<0~~~~ \forall z\in U(\xi)\text{ and } j\in \{j_1,\cdots,j_k\}.
		\end{align*}
		
		\noindent Fix $j\in \{j_1,\cdots,j_{k}\}.$ Using (\ref{E:psi_zero}), for fixed $\xi$ and $w,$ we can choose $\eps_{0,j}$ such that 
		\begin{align}\label{E:Gradient_at_Pt}
			\rl\left<\partial\psi_{j}(\xi),\bar{w}-\bar{\xi}\right>+\eps_{0,j}<0.
		\end{align}
		\noindent We write
		\begin{align}\label{E:Gradient_In_Nbd}
			\notag	\rl\left<\partial\psi_{j}(z),\bar{w}-\bar{\xi}\right>&=\rl\left<\partial\psi_{j}(z)-\partial\psi_{j}(\xi)+\partial\psi_{j}(\xi),\bar{w}-\bar{\xi}\right>\\
			&=\rl\left<\partial\psi_{j}(z)-\partial\psi_{j}(\xi),\bar{w}-\bar{\xi}\right>+\rl\left<\partial\psi_{j}(\xi),\bar{w}-\bar{\xi}\right>.
		\end{align}
		
		\noindent Since $\partial\psi_{j}$ is continuous, for $\eps=\frac{\eps_{0,j}}{\|w-\xi\|}$ there exists $\delta_{1,j}>0$ such that
		\begin{align}\label{E:Continuty_Gradient}
			\|\partial\psi_{j}(z)-\partial\psi_{j}(\xi)\|<\frac{\eps_{0,j}}{\|w-\xi\|}~~~\forall z \text{ satisfying } 0<\|z-\xi\|<\delta_{1,j}.
		\end{align}
		
		\noindent Using Cauchy-Schwarz inequality together with (\ref{E:Continuty_Gradient}), we get that 
		\begin{align*}
			\rl\left<\partial\psi_{j}(z)-\partial\psi_{j}(\xi),\bar{w}-\bar{\xi}\right> &<\|\left<\partial\psi_{j}(z)-\partial\psi_{j}(\xi),\bar{w}-\bar{\xi}\right>\|<\frac{\eps_{0,j}}{\|w-\xi\|}\|w-\xi\|<\eps_{0,j}.
		\end{align*}
		
		\noindent This implies,
		\begin{align*}
			\rl\left<\partial\psi_{j}(\xi),\bar{w}-\bar{\xi}\right>+\rl\left<\partial\psi_{j}(z)-\partial\psi_{j}(\xi),\bar{w}-\bar{\xi}\right> 
			&<\eps_{0,j}+\rl\left<\partial\psi_{j}(\xi),\bar{w}-\bar{\xi}\right><0.
		\end{align*}
		Hence
		\begin{align*}
			\rl\left<\partial\psi_{j}(z),\bar{w}-\bar{\xi}\right> <0 ~\text{ for } \|z-\xi\|<\delta_{1,j}.
		\end{align*}
		\noindent Let us now define $U(\xi):=\bigcap_{j\in\{j_1,\cdots,j_{k}\}} B_{\delta_{1,j}}(\xi),$ where $B_{\delta_{1,j}}(\xi):=\{z\in \cplx^{n}:\|z-\xi\|<\delta_{1,j}\}.$ Then for all $j\in\{j_1,\cdots,j_k\},$ we obtain
		\begin{align}\label{E:psi_zero_final}
			\sum_{l=1}^{n}\left[\frac{\partial \psi_{j}(z)}{\partial{{z_{l}}}}({w_l}-{\xi_l})+\frac{\partial \psi_{j}(z)}{\partial\bar{{z_{l}}}}(\bar{w_l}-\bar{\xi_l})\right]<0,~ \forall z\in U(\xi).
		\end{align}
		
		\noindent We take $W_{\xi}=(w-\xi).$ Now for $z\in\overline{U(\xi)}$ and $\eta>0,$ we get, by expanding in Taylor series
		\begin{align*}
			\psi_{j}(z+\eta W_{\xi})=\psi_{j}(z)+\eta\left[\sum_{l=1}^{n}\frac{\partial \psi_{j}(z)}{\partial{{z_{l}}}}({w_l}-{\xi_l})+\frac{\partial \psi_{j}(z)}{\partial\bar{{z_{l}}}}(\bar{w_l}-\bar{\xi_l})\right]	+g_{j}(\eta), 
		\end{align*}
		where $g_{j}(\eta)=o(\eta).$
		We define
		\begin{align*}
			F_{j}(z):=\sum_{l=1}^{n}\frac{\partial \psi_{j}(z)}{\partial{{z_{l}}}}({w_l}-{\xi_l})+\frac{\partial \psi_{j}(z)}{\partial\bar{{z_{l}}}}(\bar{w_l}-\bar{\xi_l}).
		\end{align*}
		Therefore,
		\begin{align}\label{E:psi_eta}
			\psi_{j}(z+\eta W_{\xi})=\psi_{j}(z)+\eta F_{j}(z)+g_{j}(\eta).
		\end{align}
		Shrinking $U(\xi),$ if needed, we get
		\begin{align*}
			\psi_{j}(z)&\le 0~~\forall z\in \overline{\Omega},
		\end{align*}
		and by  (\ref{E:psi_zero_final}), we have
		\begin{align*}
			F_{j}(z) &< 0,~\forall z\in\overline{U(\xi)}.
		\end{align*}
		Since $\overline{U(\xi)}$ is compact, there exists $\eps>0$ such that $F_{j}(z)+\eps<0.$ For this $\eps,$ we can find $\eta_{0,j}>0$ such that $|g_{j}(\eta)|<\eps\eta_{0,j}.$ Therefore, from (\ref{E:psi_eta}), we get that 
		\begin{align*}
			\psi_{j}(z+\eta W_{\xi})\le\psi_{j}(z)+\eta_{0,j}[F_{j}(z)+\eps]<0,~ \forall \eta<\eta_{0,j}.
		\end{align*}
		
		\noindent We take	$\tilde{\eta}_1:=\min\{\eta_{0,j_1},\cdots,\eta_{0,j_k}\}.$ Then we have
		\begin{align*}
			\psi_{j}(z+\eta W_{\xi})<0,~\forall \eta<\tilde{\eta}_1, \forall j\in\{j_1,\cdots,j_{k}\}, \text{ and } \forall z\in\overline{U(\xi)}\cap\overline{\Omega}.
		\end{align*}
		
		\par Next, let us fix $s\notin\{j_1,\cdots,j_k\}.$ Then $\psi_{s}(\xi)<0.$ By continuity of $\psi_{s},$ there exists a neighborhood $B_{s}(\xi)$ of $\xi$ such that 
		\begin{align*}
			\psi_{s}(z)<0,~\forall z\in B_{s}(\xi). 
		\end{align*}	
		\noindent For $z\in B_{s}(\xi)$ and for some $\delta_{1}>0,$	we have
		
		\begin{align}\label{E:psi_nonzero}
			&\notag\psi_{s}(z+\delta_1 W_{\xi})=\psi_{s}(z)+\\
			&\delta_{1}\left[\sum_{l=1}^{n}\frac{\partial \psi_{s}(z)}{\partial{{z_{l}}}}({w_l}-{\xi_l})+\frac{\partial \psi_{s}(z)}{\partial\bar{{z_{l}}}}(\bar{w_l}-\bar{\xi_l})\right]	+q_{s}(\delta_{1}), \text{ where } q_{s}(\delta_{1})=o(\delta_{1}).
		\end{align}
		\noindent	Let us define
		\begin{align*}
			G_{s}(z):=\left[\sum_{l=1}^{n}\frac{\partial \psi_{s}(z)}{\partial{{z_{l}}}}({w_l}-{\xi_l})+\frac{\partial \psi_{s}(z)}{\partial\bar{{z_{l}}}}(\bar{w_l}-\bar{\xi_l})\right].
		\end{align*}
		Therefore,	
		\begin{align*}
			\psi_{s}(z+\delta_1 W_{\xi})=\psi_{s}(z)+\delta_{1}G_{s}(z)+q_{s}(\delta_{1}).
		\end{align*}
		Since $G_{s}$ is continuous, there exists an $M>0$ such that 
		\begin{align*}
			G_{s}(z)\le M~~~\forall z\in \overline{B_{s}(\xi)}.
		\end{align*}
		Therefore, (\ref{E:psi_nonzero}) becomes
		\begin{align*}
			\psi_{s}(z+\delta_1 W_{\xi})\le\psi_{s}(z)+\delta_{1}M+q_{s}(\delta_{1}).
		\end{align*}
		Note that $\psi_{s}<0$ on $\overline{B_{s}(\xi)}$
		and $\delta_{1}M+q_{s}(\delta_{1})\to 0$ as $\delta_{1}\to 0.$ Choose $\eps>0$ such that $\psi_{s}(z)+\eps<0$ on $\overline{B_{s}(\xi)}.$ For this $\eps>0,$ there exists $\delta_{0,s}>0$ such that 
		\begin{align*}
			|\delta_{1}M+q_{s}(\delta_{1})|<\eps~~~~\text{ for } \delta_{1}<\delta_{0,s}.
		\end{align*}
		\noindent Take $B(\xi):=\displaystyle\bigcap_{s\notin\{j_1,\cdots,j_k\}}B_{s}(\xi).$
		Therefore, for all $z\in \overline{B(\xi)},$
		$\psi_{s}(z+\delta_{1}W_{\xi})<0$ for $\forall \delta_1<\delta_{0,s}.$
		Choose $\tilde{\eta}_2:=\min\{\delta_{0,s}:s\notin\{j_1,\cdots,j_k\}\},$ and  we define $\gamma:=\min\{\tilde{\eta}_1,\tilde{\eta}_2\}.$ Therefore,
		\begin{align*}
			\psi_{j}(z+kW_{\xi})<0,~\forall k<\gamma,\text{ and } \forall z\in N{(\xi)}\cap\overline{\Omega},~~\text{ where } N(\xi)\underset{\text{open}}{\subset}\overline{U(\xi)}\cap \overline{B(\xi)}.
		\end{align*}
		\noindent If we take $V_{\xi}:=\gamma W_{\xi},$ then we get from above that
		\begin{align*}
			\psi_{j}(z+tV_{\xi})<0~~~~~~\forall t, 0<t<1~~~~ \text{ and } \forall z\in N({\xi})\cap\overline{\Omega}.
		\end{align*}
		Hence, $\Omega$ has the segment property.
	\end{proof}

	\begin{corollary}\label{C:Segmnt_plyhdrn}
		Let $\mathfrak{D}_{N}$ be a bounded complex non-degenerate polynomial polyhedron. Then $\mathfrak{D}_{N}$ has the segment property.
	\end{corollary}	
	
	\begin{proof}
		We define $\psi_{j}(z)=|p_{j}(z)|^2-1:\cplx^{n}\to \mathbb{R}~~\forall j=1,\cdots,N, N\ge n.$ Then 
		\begin{align*}
			\mathfrak{D}_{N}:=\{z\in \cplx^{n}:\psi_{1}<0,\cdots,\psi_{N}<0\}.	
		\end{align*} 
		Since $\mathfrak{D}_{N}$ is a complex non-degenerate polynomial polyhedron, $dp_{j_1}\wedge\cdots\wedge dp_{j_k}(z)\ne 0 \text{ on } \Sigma_{J_k}~~\forall~ 1\le k\le N.$ Therefore, $d\psi_{j_1}\wedge\cdots\wedge d\psi_{j_k}(z)\ne 0 \text{ on } \Sigma_{J_k}~~\forall~ 1\le k\le N.$ Now by using \Cref{P:Segmnt_Prorty}, we can say that $\mathfrak{D}_{N}$ has the segment property.	
	\end{proof}

	\begin{lemma}\label{L:Plyhdrn_Star}
		Suppose that $d_1,\cdots,d_{n}$ are positive integers and $p_{j}:\cplx^{n}\to \cplx$ is a homogeneous polynomial of degree $d_{j},$ for $j=1,\cdots,n.$
		Then the \textit{polynomial polyhedron} $\overline{\mathfrak{D}}_{n}$ given by $	\mathfrak{D}_{n}:=\{z\in\cplx^n:|p_{1}(z)|<1,\cdots,|p_{n}(z)|<1\}$ is contractible. Additionally, if we take $F=(p_1,\cdots,p_{n})$ and $F^{-1}\{0\}=\{0\},$ then $\overline{\mathfrak{D}}_{n}:=\{z\in\cplx^n:|p_{1}(z)|\le1,\cdots,|p_{n}(z)|\le1\}$ is compact.
	\end{lemma}
	\begin{proof}
		Since each $p_{j}$ is a homogeneous polynomial, $0=(0,\cdots,0)\in\mathfrak{D}_{n}.$ Let $t\in[0,1]$ and $z\in \mathfrak{D}_{n}.$ Then $|p_{j}(tz)|=|t^{d_{j}}p_{j}(z)|\le|p_{j}(z)|<1.$ Therefore, $\mathfrak{D}_{n}$ is star-shaped with respect to the origin. We know that every star-shaped domain is contractible (via straight-line homotopy). 
		
		\par For the second part, assume $F=(p_1,\cdots,p_{n})$ and $F^{-1}\{0\}=\{0\}.$ By \Cref{R:Hom_Proper}, it follows that $F$ is a proper holomorphic map. Since $F^{-1}(\overline{\mathbb{D}}^n)= \overline{\mathfrak{D}}_{n},$ it follows that $ \overline{\mathfrak{D}}_{n}$ is compact, where $\overline{\mathbb{D}}^n$ is the closed unit polydisc in $\cplx^{n}.$
	\end{proof}
	
	Note that every closed bounded connected polynomial polyhedron in $\mathbb{C}$ is simply connected. However, the following example illustrates that every closed bounded polynomial polyhedron in $\mathbb{C}^2$ may not be simply connected. 
	\begin{example}
		Let $\overline{\mathfrak{D}}_2:=\{z\in \cplx^2: |z_1z_2-1|\le \frac{1}{2},|z_1+z_2-1|\le 4\}.$ Then $\mathfrak{D}_2$ is not simply connected. Note that $\overline{\mathfrak{D}}_2\subset\{z\in \cplx^2:z_{1}\ne 0\}.$ Let $\gamma:=\{(e^{i\theta},e^{-i\theta}):0\le\theta\le 2\pi\}$ be a loop in $\overline{\mathfrak{D}}_2 $ and take the projection map $\Pi:\{z\in \cplx^2:z_{1}\ne 0\}\to \cplx^{*}$ by $(z_1,z_2)\to z_1.$ Clearly, $\Pi(\gamma)=\{e^{i\theta}:0\le\theta\le 2\pi\}$ is not null homotopic in $\cplx^{*}.$ Hence, $\gamma$ is not null homotopic in $\overline{\mathfrak{D}}_2.$ Therefore, $\overline{\mathfrak{D}}_2$ is not simply connected.
	\end{example}	
	
	The proof of the following result for a $C^{1}$-smooth domain is given in
	Samuelsson and Wold \cite[Lemma 4.5]{SaW12}. 
	
	\begin{lemma}\label{L:Poly_Cnvx}
		Let $\Omega\subset\cplx^n$ be a bounded domain with $\overline{\Omega}$ polynomially convex. Also, assume that $\Omega$ has the segment property except possibly at a countable set of boundary points. Let $h_{j}\in PH(\Omega)\cap \smoo(\overline{\Omega})$ for $j=1,\cdots,N.$ Then $\Gr_{h}(\overline{\Omega})$ is polynomially convex.
	\end{lemma}
	
	\begin{proof}
		Since $\overline{\Omega}$ is polynomially convex, $\widehat{\Gr_{h}(\overline{\Omega})}\subset \overline{\Omega}\times\cplx^{N}.$ Let $(\alpha,\beta)\in \left(\overline{\Omega}\times\cplx^{N}\right)\setminus \Gr_{h}(\overline{\Omega}).$ Therefore, $\beta\ne h(\alpha).$ Then there exists $l\in\{1,\cdots,N\}$ such that $h_{l}(\alpha)\ne \beta_{l}.$ We assume that $\rl( h_{l}(\alpha)-\beta_{l})>0$. Let us define $\Psi_{l}(z,w)=\rl h_{l}(z)-\rl w_{l}.$ Clearly, $\rl h_{l}(z)$  is plurisubharmonic function on $\Omega$ and continuous on $\overline{\Omega}$. By \Cref{R:PshAprx}, the function $\rl h_{l}$ can be approximated uniformly on $\overline{\Omega}$ by a sequence of functions $\{\eta_{j}\}$ in ${\sf psh}(\overline{\Omega})\cap\smoo^{\infty}(\overline{\Omega}).$ If $\tilde{\eta}_{j}:=\eta_{j}-\rl w_{l},$ then $\{\tilde{\eta}_{j}\}$ approximates $\Psi_{l}$ uniformly on compact subsets of $\overline{\Omega}\times \cplx^{N}.$ Moreover, we have $\Psi_{l}(\alpha,\beta)>\sup_{\Gr_{h}(\overline{\Omega})}\Psi_{l}=0.$ Since $\overline{\Omega}$ is polynomially convex, there is an open Runge and Stein neighborhood $\widetilde{\Omega}\supset \overline{\Omega}$ and a function $\widetilde{\Psi_{l}}\in {\sf psh}(\widetilde{\Omega}\times \cplx^{N})$ such that $\widetilde{\Psi_{l}}(\alpha,\beta)>\sup_{\Gr_{h}(\overline{\Omega})}\widetilde{\Psi_{l}}.$ Using H\"{o}rmander's result (\Cref{R:hormander}), we get that $(\alpha,\beta)\notin\widehat{\Gr_{h}(\overline{\Omega})}_{\hol(\widetilde{\Omega}\times \cplx^{N})},$ and hence $(\alpha,\beta)\notin\widehat{\Gr_{h}(\overline{\Omega})}.$ Thus $\Gr_{h}(\overline{\Omega})$ is polynomially convex.
	\end{proof}

	Next, we state and prove a version of $Tornehave$'s result for polynomial polyhedron.	
	\begin{lemma}\label{L:Tornehave_PP}
		Let the polynomials $p_1,\cdots,p_{n}$ define the polynomial polyhedron $\mathfrak{D}_{n}$, and $\Psi=(p_1,\cdots,p_{n})$ is a proper polynomial map. Let $V$ be an irreducible analytic variety in $\mathfrak{D}_{n}$ of dimension one such that $\overline{V}\setminus V\subset \Gamma_{\mathfrak{D}_{n}}.$ Then there exists an algebraic variety $\mathcal{Z}\subset\cplx^{n}$ of dimension one such that $V=\mathcal{Z}\cap \mathfrak{D}_{n}.$
	\end{lemma}	
	
	\begin{proof}
		Let $V$ be an irreducible analytic variety of dimension one in $\mathfrak{D}_{n}$ such that $\overline{V}\setminus V\subset \Gamma_{\mathfrak{D}_{n}}.$ By \Cref{R:Remmert}, $\Psi(V)$ is also an irreducible analytic variety of dimension one in $\mathbb{D}^n.$ It is easy to show that $\overline{\Psi(V)}\setminus \Psi(V)\subset \mathbb{T}^{n}.$ Then by \Cref{R:Tornehave_BiDisc}, there exists an algebraic variety $Z\subset\cplx^{n}$ such that $\Psi(V)=Z\cap \mathbb{D}^{n}.$ Therefore, $\Psi^{-1}(\Psi(V))=\Psi^{-1}(Z\cap \mathbb{D}^{n})=\Psi^{-1}(Z)\cap \Psi^{-1}(\mathbb{D}^{n}).$ Since $Z$ is an algebraic variety in $\cplx^{n}$ and $\Psi$ is a polynomial map, $\Psi^{-1}(Z)$ is also an algebraic variety in $\cplx^{n}.$ Assume that $\Psi^{-1}(Z)=\cup Z_{j},$ where $Z_{j}$ is an irreducible component of $\Psi^{-1}(Z)$ of dimension one.  Clearly, $V\subset\Psi^{-1}(Z)=\cup Z_{j}.$ We claim that $V\subset Z_{j_0}$ for some $j_{0}.$ Without loss of generality, we assume that $V \nsubseteq Z_1, V  \nsubseteq Z_2$ but $V\subset Z_{1}\cup Z_{2}.$ Hence $V=V\cap (Z_{1}\cup Z_{2})=Z_{1}\cap V\cup Z_2\cap V.$ Therefore, $V\setminus (Z_{1}\cap Z_{2})=(Z_{1}\cap V\cup Z_2\cap V)\setminus(Z_{1}\cap Z_{2}).$ This implies,
		\begin{align}\label{E:Regular set_V}
			V\setminus (Z_{1}\cap Z_{2})=((Z_{1}\cap V) \setminus(Z_{1}\cap Z_{2}))\cup ((Z_2\cap V)\setminus(Z_{1}\cap Z_{2})).
		\end{align}
		We set, $S:=V_{Sing}\cap (Z_{1}\cap Z_{2}).$  Note that every point of $Z_1\cap Z_2$ is a singular point for $Z_1\cup Z_2$, and hence isolated and regular point of the variety $V$ can not be isolated. Here $V$ is one-dimensional irreducible variety. Therefore, no point of $Z_1\cap Z_2$ is a regular point of $V.$ Hence we obtain that, $V\setminus S=V\setminus (Z_{1}\cap Z_{2}).$ Therefore, from (\ref{E:Regular set_V}) we get that
		\begin{align}\label{E:SinSet_connctd}
			V\setminus S=((Z_{1}\cap V) \setminus(Z_{1}\cap Z_{2}))\cup ((Z_2\cap V)\setminus(Z_{1}\cap Z_{2})).
		\end{align}	
		Since $S\subset V_{Sing}$ and the set of regular points of an irreducible variety is connected, therefore we can say that $V\setminus S$ is connected. But right hand side of (\ref{E:SinSet_connctd}) is disconnected, and this is a contradiction to the assumption that $V \nsubseteq Z_1, V  \nsubseteq Z_2$ but $V\subset Z_{1}\cup Z_{2}.$ Without loss of generality, we assume that $V\subset Z_{1}.$
		
		\smallskip
		
		\noindent  Next, we claim that $V=\mathfrak{D}_{n}\cap Z_{1}.$ Assume that $V\subset Z_{1}\cap \mathfrak{D}_{n}.$  Since $Z_{1}$ is an irreducible variety and if $V$ is proper subset of $Z_{1},$ then $\dim_{\mathbb{C}} V<\dim_{\mathbb{C}} Z_{1} $ (see Izzo \cite{AIW2}, Federer \cite[Theorem 3.4.8, assertion (15)]{Feder69}). This is not possible, as otherwise, $V$ will be discrete. Therefore, $V=\mathfrak{D}_{n}\cap Z_{1}.$	
	\end{proof}

	The proof of the following result for a compact connected subset of $\cplx^n$ is given in Jimbo \cite[Main Lemma]{Jimbo03}. 
	\begin{lemma}\label{L:Tot_rlpt_Pk pt}
		Let $X$ be compact subset of $\cplx^n.$ Let $\Omega$ be an open subset of $\cplx^n$ such that $X\cap \Omega=\emptyset.$ If $\widehat{X}\cap \Omega$ is contained in a totally real manifold $M$ of $\Omega,$ then $\widehat{X}\cap \Omega=\emptyset.$
	\end{lemma}

	\begin{theorem}\label{T: Approx_Cont_Func}
		Let $p_1,\cdots,p_n$ be holomorphic polynomials in $\cplx^{n}$ such that $z\mapp (p_1(z),\cdots,p_{n}(z))$ is a proper holomorphic map from $\cplx^{n}$ to $\cplx^{n}.$ Define $\Gamma:=\{z\in\cplx^n:|p_1(z)|=1,\cdots,|p_{n}(z)|=1\}$ and assume that $dp_{1}\wedge\cdots\wedge dp_{n}(z)\ne 0 \text{ on } \Gamma.$ Let $N\ge 1$ and $f_1,\cdots,f_{N}\in \smoo(\Gamma).$ Then $[z_1,\cdots,z_n,f_{1},\cdots,f_N;\Gamma]=\smoo(\Gamma)$ if and only if $\Gr_{f}(\Gamma)$ is polynomially convex, where $f=(f_1,\cdots,f_{N}).$  
	\end{theorem}	
	
	\begin{proof}
		Since the map $z\mapp (p_1(z),\cdots,p_n(z))$ is proper map from $\cplx^{n}$ to $\cplx^{n},$ $\Gamma$ is compact. First, we show that $[z_1,\cdots,z_n,\overline{{p_1}},\cdots,\overline{{p_n}};\Gamma]=\smoo(\Gamma).$ We consider 
		\begin{align*}
			K:&=\{(z_1,\cdots,z_n,w_1,\cdots,w_n):w_{j}=\overline{p_{j}(z)}, |p_{j}(z)|=1,j=1,\cdots,n\}\\
			&=\{(z_1,\cdots,z_n,w_1,\cdots,w_n):w_{j}p_{j}(z)=1, |p_{j}(z)|=1,j=1,\cdots,n\}\subset \cplx^{2n}.
		\end{align*}
		We claim that $K$ is polynomially convex. To prove this claim, we assume that $(a,b)\in \cplx^{2n}\setminus K,$ where $a=(a_1,\cdots,a_n)$ and $b=(b_1,\cdots,b_n).$	
		
		\noindent  {\bf Case I:} $b_{l}p_{l}(a)\not =1,$ $|p_{l}(a)|=1$ for some $l\in \{1,\cdots,n\}.$ Consider the polynomial $\psi(z,w)=w_{l}p_{l}(z)-1.$ Then
		\begin{align*}
			|\psi(a,b)|>0=\sup_{K}|\psi|.
		\end{align*}
		Hence, in this case, $(a,b)\notin\widehat{K}.$
		
		\noindent  {\bf Case II:} $b_{l}p_{l}(a) =1,$ $|p_{l}(a)|<1$ for some $l\in \{1,\cdots,n\}.$ This implies $|b_{l}|>1.$ Consider the polynomial $\psi(z,w)=w_{l}.$ Then
		\begin{align*}
			|b_{l}|=|\psi(a,b)|>1=\sup_{K}|\psi|.
		\end{align*}
		Hence, in this case, $(a,b)\notin\widehat{K}.$
		
		\noindent  {\bf Case III:} $b_{l}p_{l}(a) =1,$ $|p_{l}(a)|>1$ for some $l\in \{1,\cdots,n\}.$ Consider the polynomial $\psi(z,w)=p_{l}(z).$ Then
		\begin{align*}
			|p_{l}(a)|=|\psi(a,b)|>1=\sup_{K}|\psi|.
		\end{align*}
		Hence, in this case, $(a,b)\notin\widehat{K}.$
		
		\noindent Combining all three cases, we conclude that $K$ is polynomially convex. Since $dp_{1}\wedge\cdots\wedge dp_{n}(z)\ne 0 \text{ on } \Gamma,$ $\Gamma$ is totally real, and being graph over a totally real manifold, $K$ is also totally real. Therefore, we get that $\poly(K)=\smoo(K).$ Since $K$ is the graph of $(\overline{p_1},\cdots,\overline{p_n})$ on $\Gamma,$ hence
		\begin{align}\label{E:Apprx_Poly}
			[z_1,\cdots,z_n,\overline{p_{1}},\cdots,\overline{p_n};\Gamma]=\smoo(\Gamma).
		\end{align}
		Now we will prove the main result. Assume $[z_1,\cdots,z_n,f_{1},\cdots,f_N;\Gamma]=\smoo(\Gamma).$ Define $X=\Gr_{f}(\Gamma).$ Since $[z_1,\cdots,z_n,f_{1},\cdots,f_N;\Gamma]=\smoo(\Gamma),$ then $\poly(X)=\smoo(X).$ Hence $X$ is polynomially convex.
		
		\smallskip 
		
		\noindent Conversely, assume that $\Gr_{f}(\Gamma)$ is polynomially convex. Let $\Omega$ be a neighborhood of $\Gamma$ such that $p_{1}(z)\not =0~~\text{ on } \Omega.$ Define $\phi(z,w)=\frac{1}{p_1(z)}.$ Then $\phi$ is holomorphic on $\Omega.$ Also, $\phi$ is holomorphic on $\Omega\times \mathbb{C}^{n}.$ Since $\Gr_{f}(\Gamma)\subset \Omega\times\cplx^n,$ by the \textit{Oka-Weil} approximation theorem, there exists a sequence of polynomials $P_{j}$ in $\cplx^{n}_{z}\times \cplx^{n}_{w}$ such that
		$P_{j}\to \phi$  uniformly on $ \Gr_{f}(\Gamma).$ This implies $P_{j}(z,f(z))\to \phi(z,w)=\frac{1}{p_{1}(z)}=\overline{p_{1}(z)}$ uniformly on  $\Gamma.$ Hence $\overline{p_{1}(z)}\in [z_1,\cdots,z_n,f_{1},\cdots,f_N;\Gamma].$ By the similar method we can show that $\overline{{p_j}}\in [z_1,\cdots,z_n,f_{1},\cdots,f_N;\Gamma]~\forall_{j}\in\{1,\cdots,n\}.$ Hence, $[z_1,\cdots,z_n,\overline{p_{1}},\cdots,\overline{p_n};{\Gamma}]\subset [z_1,\cdots,z_n,f_{1},\cdots,f_N;\Gamma].$ Using (\ref{E:Apprx_Poly}), we get that 
		\begin{align*}
			[z_1,\cdots,z_n,\overline{p_{1}},\cdots,\overline{p_n};\Gamma]  =\smoo(\Gamma) =[z_1,\cdots,z_n,f_{1},\cdots,f_N;\Gamma].
		\end{align*}
	\end{proof}

	\section{Proof of \Cref{T:PPolyhrn_DistingBdry}}\label{S:proofMain}
	
	Let $h_{j}$ denote the pluriharmonic extension of $f_{j}$ to $\mathfrak{D}_{2},$ and write $h=(h_1,\cdots,h_{N}):\overline{\mathfrak{D}}_{2}\to \cplx^{N}.$ The following two cases may occur:
	
	\smallskip
	
	\noindent {\bf Case I}: $\Gr_{h}(\Gamma_{\mathfrak{D}_{2}})$ is polynomially convex. In view of \Cref{T: Approx_Cont_Func}, we have that \begin{align*}
		[z_1,z_2,f_1,\cdots,f_{N};\Gamma_{\mathfrak{D}_{2}}]=\smoo(\Gamma_{\mathfrak{D}_{2}}).
	\end{align*} 
	Therefore, in this case,  the proof of the theorem is complete.
	
	\smallskip
	
	\noindent {\bf Case II}:  $\Gr_{h}(\Gamma_{\mathfrak{D}_{2}})$ is not polynomially convex. By \Cref{L:Tornehave_PP}, it is enough to find an irreducible analytic variety $V\subset \overline{\mathfrak{D}}_{2}\setminus \Gamma_{\mathfrak{D}_{2}}$ such that $\overline{V}\setminus V\subset \Gamma_{\mathfrak{D}_{2}}.$ 
	
	\smallskip
	
	\noindent Since $\Gr_{h}(\Gamma_{\mathfrak{D}_{2}})$ is not polynomially convex, there exists $\alpha \in \overline{\mathfrak{D}}_{2}\setminus \Gamma_{\mathfrak{D}_{2}}$ such that $\Gr_{h}(\alpha) \in \widehat{\Gr_{h}(\Gamma_{\mathfrak{D}_{2}})}.$ Now, we discuss different situations through some lemmas.

	\begin{lemma}\label{L:Lemma1}
		If there exists $\alpha\in \partial{\mathfrak{D}_{2}}-\Gamma_{\mathfrak{D}_{2}}$ such that $\Gr_{h}(\alpha) \in \widehat{\Gr_{h}(\Gamma_{\mathfrak{D}_{2}})},$ then there exists an analytic set $\triangle\subset\left(\partial{\mathfrak{D}_{2}}-\Gamma_{\mathfrak{D}_{2}}\right) $ where all $h_{j}$'s are holomorphic.
	\end{lemma}
	\begin{proof}
		Since $\alpha\in \partial{\mathfrak{D}_{2}}-\Gamma_{\mathfrak{D}_{2}},$ without loss of generality, we assume that $\alpha\in \{z\in \cplx^2:|p_1(z)|=1,|p_2(z)|<1\}.$ We now consider the set
		\begin{align*}
			\triangle_{\alpha}&:=\{(z_1,z_2)\in \mathfrak{D}_{2}: p_1(z_1,z_2)=p_{1}(\alpha), |p_2(z_1,z_2)|<1 \}\subset\partial{\mathfrak{D}_{2}},\\
			\partial\triangle_{\alpha}&=\{(z_1,z_2)\in \mathfrak{D}_{2}:p_{1}(z_1,z_2)=p_{1}(\alpha),|p_2(z_1,z_2)|=1\}.
		\end{align*}
		
		\noindent We claim that $(\alpha,h(\alpha))\in \widehat{\\Gr_{h}(\partial\triangle_{\alpha})}:$ if not, then there exists a polynomial $\rho(z,w)$ in $\cplx^{2+N}$ such that $\rho(\alpha,h(\alpha))=1$ and $\sup_{\Gr_{h}({\partial\triangle_{\alpha}})}|\rho(z,w)|<1,$ i.e. $\sup_{\partial\triangle_{\alpha}}\left|\rho(z,h(z))\right|<1.$ Next, we consider the function $\psi(z_1,z_2)=\frac{1}{2}\left(1+p_{1}(z_1,z_2)\overline{p_{1}(\alpha)}\right).$ Note that 
		\begin{align*}
			\psi&=1 \text{ on } {\partial\triangle_{\alpha}}\\
			|\psi|&<1 \text{ on } \Gamma_{\mathfrak{D}_{2}}\setminus {\partial\triangle_{\alpha}}.
		\end{align*}
		\noindent In fact,
		\begin{align*}
			\psi(z_1,z_2)&=\frac{1}{2}\left(1+p_{1}(z_1,z_2)\overline{p_{1}(\alpha)}\right)=\frac{1}{2}\left(p_1(\alpha)\overline{p_1(\alpha)}+p_{1}(z_1,z_2)\overline{p_{1}(\alpha)}\right).
		\end{align*}
		\noindent This implies
		\begin{align*}
			|\psi(z_1,z_2)|&=\frac{1}{2}\left|p_{1}(\alpha)+p_{1}(z_1,z_2)\right|.
		\end{align*}	
		\noindent Let $\omega=p_{1}(\alpha)+p_{1}(z_1,z_2).$ We will find $\max |\omega|$ on $\Gamma_{\mathfrak{D}_{2}},$ i.e., we need to find $\max |\omega|$ on $|\omega-p_{1}(\alpha)|=|p_{1}(z)|=1.$ Let $\omega=x+iy~~\text{ and } p_{1}(\alpha)=\lambda_1+i\lambda_2.$ The equation of the straight line through the origin and through the centre of the circle $|\omega-p_{1}(\alpha)|=1$ is given by $y=\frac{\lambda_2}{\lambda_1}x.$ Solving $y=\frac{\lambda_2}{\lambda_1}x$ and $(x-\lambda_1)^2+(y-\lambda_2)^2=1,$ we get that $(x,y)=(0,0),~~\text{ or } (x,y)=(2\lambda_1,2\lambda_2).$ Therefore, $\max|\omega|=2$ and maximum attains at $\omega=2\lambda_1+i2\lambda_2=2p_1(\alpha).$ Then $\omega=p_{1}(\alpha)+p_{1}(z_1,z_2),$ implies $p_1(z)=p_1(\alpha).$ Therefore, $\max |\psi|=1$ on $\Gamma_{\mathfrak{D}_{2}}$	and since $|\omega|<1$ if $\omega\not =2p_{1}(\alpha),$ $|\psi|<1 \text{ on } \Gamma_{\mathfrak{D}_{2}}\setminus \partial\triangle_{\alpha}.$
		
		\noindent Finally we define the polynomial in $\cplx^{2+N}$
		\begin{align*}
			\chi(z,w):=\psi^{m}(z)\times\rho(z,w).
		\end{align*}
		Then 
		\begin{align*}
			\chi(\alpha,h(\alpha))&=1~~~ \text{ and }\\
			\sup_{\Gr_{h}\left(\Gamma_{\mathfrak{D}_{2}}\right)}|\chi(z,w)|&=\left|\chi(\eta,h(\eta))\right| \text{ for some } \eta\in \Gamma_{\mathfrak{D}_{2}}.
		\end{align*}
		If $\eta\in {\partial\triangle_{\alpha}},$ then $\left|\chi(\eta,h(\eta))\right|<1.$ If $\eta\notin {\partial\triangle_{\alpha}},$ since  $|\psi|<1$ on $\Gamma_{\mathfrak{D}_{2}}\setminus {\partial\triangle_{\alpha}},$ for sufficiently large $m$ we can make $\left|\chi(\eta,h(\eta))\right|<1.$ Therefore, $(\alpha,h(\alpha))\notin \widehat{\Gr_{h}(\Gamma_{\mathfrak{D}_{2}})}.$ This is a contradiction. Hence, our claim is true. This implies $\Gr_{h}(\partial\triangle_{\alpha})$ is not polynomially convex. Note that $\{z\in \cplx^2:p_{1}(z)=1\}$ is equivalent to $\cplx$ (by \Cref{Rmk:Abhy_Sukuki}), and since $\overline{\triangle_{\alpha}}$ is polynomially convex, each component of $\triangle_{\alpha}$ is simply connected. Also, by using \Cref{L:Harmonic_bdry}, we have that $h_{j}$'s are harmonic on $\triangle_{\alpha}.$ Then by applying the Wermer maximality theorem, we obtain that all $h_{j}$'s are holomorphic on some component of $\triangle_{a}$ and $\Gr_{h}(\triangle_{\alpha})\subset \widehat{\Gr_{h}(\Gamma_{\mathfrak{D}_{2}})}.$
	\end{proof}	
	
	\noindent From next, we assume that
	
	\begin{align}\label{E:Hull_bdry_empty}
		\widehat{\Gr_{h}(\Gamma_{\mathfrak{D}_{2}})}\cap \Gr_{h}\left(\partial{\mathfrak{D}_{2}}-\Gamma_{\mathfrak{D}_{2}}\right)=\emptyset.
	\end{align}
	
	We set 
	\begin{align}\label{E:Set_Graph_Complx_Tngnt}
		\widetilde{V}:=\{z\in \mathfrak{D}_{2}:\bar{\partial}{{h_{j_1}}}\wedge\bar{\partial}{{h_{j_2}}}(z)=0, \forall~ {j_1,j_2}, 1\le j_1,j_2\le N\}.
	\end{align}
	Because of \Cref{R:Coplx_Pt_Variety} and using the fact that the finite intersection of analytic varieties is again an analytic variety, we get that $\widetilde{V}$ is an analytic variety in $\mathfrak{D}_2$.
	
	\noindent By \Cref{R:Graph_Tot_Rl}, we get that $\Gr_{h}(\mathfrak{D}_{2}\setminus\widetilde{V})$ is totally real. Let $E=\Gr_{h}(\Gamma_{\mathfrak{D}_{2}})$ and $\Omega=(\mathfrak{D}_{2}\setminus\widetilde{V})\times\cplx^{N}.$ Then 
	\begin{align*}
		\widehat{E}\cap \Omega\subset\left(\Gr_{h}(\overline{\mathfrak{D}}_{2})\cap(\mathfrak{D}_{2}\setminus\widetilde{V})\times\cplx^{N}\right)\subset \Gr_{h}(\mathfrak{D}_{2}\setminus \widetilde{V}).
	\end{align*}
	
	\noindent Since $\Gr_{h}(\mathfrak{D}_{2}\setminus \widetilde{V})$ is totally real, by \Cref{L:Tot_rlpt_Pk pt}, we get that $\widehat{E}\cap \Omega=\emptyset.$ Hence
	\begin{align}\label{E:Hull_Tot_Rl_empty}
		\widehat{\Gr_{h}(\Gamma_{\mathfrak{D}_2})}\cap\Gr_{h}(\mathfrak{D}_{2}\setminus\widetilde{V})\subset \widehat{\Gr_{h}(\Gamma_{\mathfrak{D}_2})}\cap(\mathfrak{D}_{2}\setminus\widetilde{V})\times \cplx^{N}=\emptyset.
	\end{align} 
	
	\noindent In view of (\ref{E:Hull_bdry_empty}) and (\ref{E:Hull_Tot_Rl_empty}), we get that
	\begin{align*}
		\widehat{\Gr_{h}(\Gamma_{\mathfrak{D}_{2}})}\subset \Gr_{h}(\Gamma_{\mathfrak{D}_{2}}\cup\widetilde{V}),
	\end{align*}
	i.e.
	\begin{align}\label{E:Hull_disbndry}
		\widehat{\Gr_{h}(\Gamma_{\mathfrak{D}_{2}})}\setminus \Gr_{h}(\Gamma_{\mathfrak{D}_{2}}) \subset \Gr_{h}(\widetilde{V}).
	\end{align}
	\noindent  Note that $\widetilde{V}$ can not be discrete. Therefore, two cases can occur: either $\widetilde{V}\ne \mathfrak{D}_{2}$ or $\widetilde{V}=\mathfrak{D}_{2}.$ First, we consider the case $\widetilde{V}\ne \mathfrak{D}_{2}.$

	\begin{lemma}\label{L:Lemma2}
		Assume that $\widetilde{V}\not =\mathfrak{D}_{2}.$	Let $V$ be an irreducible component of $\widetilde{V}$ and let $\alpha\in V\cap (\widetilde{V})_{\text{reg}}$ with $\Gr_{h}(\alpha)\in \widehat{\Gr_{h}(\Gamma_{\mathfrak{D}_{2}})}.$ Then $\Gr_{h}(V)\subset \widehat{\Gr_{h}(\Gamma_{\mathfrak{D}_{2}})}$ and all $h_{j}$'s are holomorphic on $V.$
	\end{lemma}
	\begin{proof}
		We choose a neighborhood $U$ of $\alpha$ in $V$ such that $U\cap V=:D_{\alpha}$ is a smoothly bounded disc containing $\alpha.$ Let us define $X:=\widehat{\Gr_{h}(\Gamma_{\mathfrak{D}_{2}})}\cap(\partial U\times\cplx^{N}).$ Note that $X$ is a closed subset of a compact set $\widehat{\Gr_{h}(\Gamma_{\mathfrak{D}_{2}})},$ and hence $X$ is compact. Since $\widehat{\Gr_{h}(\Gamma_{\mathfrak{D}_{2}})}\subset \Gr_{h}(\Gamma_{\mathfrak{D}_{2}}\cup \widetilde{V}),$ we have 
		\begin{align*}
			X\subset \Gr_{h}(\Gamma_{\mathfrak{D}_{2}}\cup \widetilde{V})\cap (\partial U\times\cplx^{N})\subset \Gr_{h}(\partial U\cap V).
		\end{align*}

		We claim that		
		\begin{align*}
			\Gr_{h}(\alpha)\in \widehat{X}\subset \widehat{\Gr_{h}(\partial U\cap V)}= \widehat{\Gr_{h}(\partial D_{\alpha})}.
		\end{align*}
		
		\noindent To prove this claim, let $E=\Gr_{h}(\Gamma_{\mathfrak{D}_{2}})$ and $N=U\times\mathbb{B}(r)$ with $r>\sup_{\overline{U}}|h|.$ Note that $E\cap N=\emptyset$ and $\partial N\cap\widehat{E}=[(\partial U\times\mathbb{B}(r))\cap\widehat{E}]\cup [(U\times \partial\mathbb{B}(r))\cap\widehat{E}].$ We claim that $(U\times \partial\mathbb{B}(r))\cap\widehat{E}=\emptyset:$ if possible, let $(\eta,\xi)\in (U\times \partial\mathbb{B}(r))\cap\widehat{E}.$ Then $\xi=h(\eta)$ and $\eta\in U, |\xi|=r.$ But on $U,$ we have $|\xi|=|h(\eta)|<r.$ This is a contradiction to the assumption that $(U\times \partial\mathbb{B}(r))\cap\widehat{E}\not =\emptyset.$ Then by \Cref{R:Rossi}, we have
		\begin{align*}
			\Gr_{h}(\alpha)\in \widehat{\widehat{E}\cap \partial N}&=\widehat{[\widehat{\Gr_{h}(\Gamma_{\mathfrak{D}_{2}})}\cap (\partial U\times\mathbb{B}(r))]}\\
			&\subset \widehat{[\widehat{\Gr_{h}(\Gamma_{\mathfrak{D}_{2}})}\cap (\partial U\times\cplx^{N})]}=\widehat{X}\subset \widehat{\Gr_{h}(\partial U\cap V)}= \widehat{\Gr_{h}(\partial D_{\alpha})}.
		\end{align*}
		
		\noindent Therefore, $\widehat{\Gr_{h}(\partial D_{\alpha})}$ is not polynomially convex, and hence by the Wermer maximality theorem we can say that all $h_{j}$'s are holomorphic in $D_{\alpha}.$ Note that $X=\Gr_{h}(\partial D_{\alpha}).$ In fact, if $X$ is proper subset of $\Gr_{h}(\partial D_{\alpha}),$ then $X=\Gr_{h}(K)$ for some proper compact subset of $\partial D_{\alpha}$ and in this case $X$ will be polynomially convex. Hence $\Gr_{h}(\alpha)\in \widehat{X}.$ This is a contradiction and hence $X=\Gr_{h}(\partial D_{\alpha}).$ It is easy to see that $\Gr_{h}(\overline{D_{\alpha}})\subset \widehat{X}.$ To see this, let $(a,h(a))\in \Gr_{h}(\overline{D_{\alpha}}),$ and $P$ be any holomorphic polynomial in $\cplx^2\times \cplx^{N}.$ Then 
		\begin{align*}
			|P((a,h(a)))|\le& \sup_{\Gr_{h}(\overline{D_{\alpha}})}|P(z,w)|\\
			=&\sup_{(\partial D_{\alpha})}|P(z,h(z))|~ (\text{since } h \text{ is holomorphic on}~ D_{\alpha}).
		\end{align*}
		\noindent Hence $(a,h(a))\in \widehat{X}.$ Therefore, $\Gr_{h}(D_{\alpha})\subset \widehat{X}\subset \widehat{ \Gr_{h}(\Gamma_{\mathfrak{D}_{2}})}$ and this holds near every regular point. Since regular points are dense, it follows that if $\beta$ is a singular point of $V,$ we have  $\Gr_{h}(\beta)\in \widehat{\Gr_{h}(\Gamma_{\mathfrak{D}_{2}})}$ also true. Hence $\Gr_{h}(V)\subset\widehat{ \Gr_{h}(\Gamma_{\mathfrak{D}_{2}})},$ and in view of (\ref{E:Hull_bdry_empty}), we have that $\overline{V}\setminus V\subset \Gamma_{\mathfrak{D}_{2}}.$
	\end{proof}

	\noindent From now, we assume that $$\widetilde{V}=\mathfrak{D}_{2}.$$

	\noindent In view of \Cref{L:Plurihrmonic_Cnjgte}, we can say that there exist holomorphic functions $\Phi_{j}$ and $\Psi_{j}$ on $\mathfrak{D}_{2}$ such that 
	\begin{align}\label{E:Change h to u}
		\imag h_{j}=\rl \Psi_{j} \text{ and } u_{j}:=\rl \Phi_{j}=h_{j}-\Psi_{j}.
	\end{align}
	
	\noindent We use the following change of coordinates $\widetilde{\Psi}:\mathfrak{D}_2\times\cplx^{N}\mapsto\mathfrak{D}_2\times\cplx^{N}$ on $\mathfrak{D}_2\times \cplx^N$ defined by
	\begin{align*}
		\widetilde{\Psi}(z_1,z_2,w_{1},\cdots,w_{N})=(z_1,z_{2},w_{1}-\Psi_{1}(z),\cdots,w_{N}-\Psi_{N}(z)).
	\end{align*}
	We set $u=(u_1,u_2,\cdots,u_{N})$ and $\Psi(z)=\left(\Psi_1(z),\cdots,\Psi_N(z)\right).$ Let $K$ be any compact subset  of $\mathfrak{D}_2.$ Then  $\widetilde{\Psi}\left(\widehat{\Gr_{h}}(K)\right)=\widehat{\Gr_{u}(K)}.$ Therefore, we have the following: 
	\begin{lemma}\label{L:biholo_cordnte}
		Let $K$ be any compact subset  of $\mathfrak{D}_2.$ Then  $\widetilde{\Psi}\left(\widehat{\Gr_{h}}(K)\right)=\widehat{\Gr_{u}(K)}.$ Therefore, a point $\Gr_{h}(\alpha)\in \widehat{\Gr_{h}(K)}$ is equivalent to $\Gr_{u}(\alpha)\in \widehat{\Gr_{u}(K)}.$ 
	\end{lemma}

	\medskip

	By (\ref{E:Hull_bdry_empty}) and the fact that $\widetilde{V}=\mathfrak{D}_{2},$ we get that
	
	\begin{align*}
		\widehat{\Gr_{h}(\Gamma_{\mathfrak{D}_{2}})}\subset  \Gr_{h}\left(\Gamma_{\mathfrak{D}_{2}}\right)\cup \Gr_{h}\left({\mathfrak{D}_{2}}\right).
	\end{align*}
	Therefore $\widehat{\Gr_{h}(\Gamma_{\mathfrak{D}_{2}})}=\Gr_{h}(D),$ for some subset $D$ of $\overline{\mathfrak{D}_{2}}.$ Since polynomial convex hull of a compact set is compact, $D$ is also compact. Let $r_{j}$ be an increasing sequence of real numbers such that $r_{j}\to 1$ as $j\to \infty.$ We define
	\begin{align*}
		R^{1}_{j}&:=\min\left\{|p_{1}(z)|:\{|p_{1}(z)|\le r_j,|p_{2}(z)|=r_{j}\}\cap D\ne \emptyset\right\}\\
		R^{2}_{j}&:=\min\left\{|p_{2}(z)|:\{|p_{2}(z)|\le r_j,|p_{1}(z)|=r_{j}\}\cap D\ne \emptyset\right\}.
	\end{align*}
	\noindent Observe that $R^{1}_{j}=p_{1}(\xi_{j})$ for some $\xi_{j}\in D,$ and $\xi_{j}\to \Gamma_{\mathfrak{D}_{2}}$ as $j\to \infty.$ Therefore, $j\to \infty$ implies $R^{1}_{j}\to 1.$ Similarly, $R^{2}_{j}\to 1$ as $j\to \infty.$ Define a sequence of real numbers
	\begin{align*}
		\eps_j:=\max\left\{r_{j}-R^{1}_{j},r_j-R^{2}_{j}\right\}. 
	\end{align*} 
	Since $r_{j}-R^{1}_{j},r_{j}-R^{2}_{j}\to 0$ as $j\to \infty,$ therefore, $\eps_{j}\to 0$ as $j\to \infty.$ Define		
	\begin{align*}
		\mathfrak{Q}_{j}:=\{z\in \mathfrak{D}_{2}:|p_1|=r_j, r_j-\eps_{j}\le |p_2|\le r_{j}\}\cup \{z\in \mathfrak{D}_{2}:|p_2|=r_j, r_j-\eps_{j}\le |p_1|\le r_{j}\}.
	\end{align*} and
	\begin{align*}
		\overline{\mathfrak{D}^{r_j}_{2}}=\{z\in \mathfrak{D}_{2}:|p_2|\le r_j, |p_1|\le r_{j}\}.
	\end{align*}
	Then from the above construction, we obtained that
	\begin{align}\label{E:hull_Qj}
		\widehat{\Gr_{h}(\Gamma_{\mathfrak{D}_{2}})}\cap \Gr_{h}(\partial\mathfrak{D}^{r_j}_{2})\subset \Gr_{h}(\mathfrak{Q}_{j}).
	\end{align}
	
	\begin{lemma}\label{L:Lemma3}
		Assume that there is a subsequence $r_{j_k}$ of $r_j$ such that for each $k$ there exists a point $\Gr_{h}(a_k,b_k)\in\widehat{ \Gr_{h}(\mathfrak{Q}_{j_k})}$ with $|p_{1}(a_k,b_k)|=r_{j_k}$ and $|p_2(a_k,b_k)|<r_{j_k}-\eps_{j_k}.$ Then there exists a analytic set $\triangle$ and all $h_{j}$'s are holomorphic on $\triangle.$
	\end{lemma}
	\begin{proof}
		Let $\xi\in \{|p_{1}|(z)=r_{j_k},r_{j_k}-\eps_{j_k}\le|p_{2}(z)|\le r_{j_k}\}\subset \mathfrak{Q}_{j_k}.$ Then $|p_{1}(\xi)|=r_{j_k}.$ Also $|p_{1}(a_k,b_k)|=r_{j_k}.$ Therefore, $p_{1}(\xi)=e^{i\theta_{0}}p_{1}(a_k,b_k)$ for some $\theta_{0}.$ For each $k,$ we consider the set
		\begin{align*}
			\triangle_{k}:=\{z:p_1(z)=e^{i\theta_{0}}p_1(a_k,b_k), r_{j_k}-\eps_{j_k}\le |p_2(z)|\le r_{j_k}\}.	
		\end{align*}

		\noindent We claim that $\Gr_{h}(a_k,b_k)\in\widehat{ \Gr_{h}(\triangle_{k})}:$ if not, then there exists a polynomial $\digamma_{k}$ such that 
		\begin{align*}
			\notag	\digamma_{k}((a_k,b_k),h(a_k,b_k))&=1 \text{ and }
			\sup_{\Gr_{h}(\triangle_{k})}|\digamma_{k}(z,w)|<1.
		\end{align*}
		
		\noindent This implies
		\begin{align}\label{E:Sup_annulus}
			\sup_{\triangle_{k}}|\digamma_{k}(z,h(z))|&<1.
		\end{align}

		\noindent Again, consider the polynomial 
		\begin{align}\label{E:Pk_pt_Anunulus}
			\Upsilon_{k}(z_1,z_2):=\frac{1}{2}\left[1+\frac{p_1(z_1,z_2)e^{-i\theta_{0}}\overline{p_{1}(a_k,b_k)}}{r^2_{j_k}}\right].
		\end{align}
		We observe that $\Upsilon_{k}=1 \text{ if } p_{1}(z)=e^{i\theta_{0}}p_1(a_k,b_k)$ and
		\begin{align*}
			|\Upsilon_{k}|&<1 \text{ if } p_{1}(z)\not =e^{i\theta_{0}}p_1(a_k,b_k)~~(\text{by calculations same as in the proof of \Cref{L:Lemma1}}).
		\end{align*}
		Finally, consider the polynomial
		\begin{align*}
			\kappa_{k}(z,w)=\digamma_{k}(z,w) \Upsilon^{m}_{k}(z)
		\end{align*}
		Hence $\kappa_{k}\left((a_k,b_k),h(a_k,b_k)\right)=1.$
		We claim that			
		\begin{align*}
			\sup_{\Gr_{h}(\mathfrak{Q}_{j})}|\kappa_{k}(z,w)|.
		\end{align*}
		To see a proof of the claim, let $\sup_{\Gr_{h}(\mathfrak{Q}_{j})}|\kappa_{k}(z,w)|=|\kappa_{k}(s,h(s))|,$ for some $s=(s_1,s_2)\in \mathfrak{Q}_{j}.$ If $s\in \triangle_{k},$ then from (\ref{E:Sup_annulus}) and (\ref{E:Pk_pt_Anunulus}), we get that $|\kappa_{k}(s,h(s))|<1.$ If $s\in \mathfrak{Q}_{j}\setminus \triangle_{k},$ then for large enough $m,$ we have that $|\kappa_{k}(s,h(s))|<1.$ 
		This is a contradiction. Hence $\Gr_{h}(a_k,b_k)\in\widehat{ \Gr_{h}(\triangle_{k})}.$ This implies $\Gr_{u}(a_k,b_k)\in\widehat{ \Gr_{u}(\triangle_{k})},$ where $u$ is given by \ref{E:Change h to u}. Note that $\{p_{1}(z)=e^{i\theta_{0}}p_{1}(a_k,b_{k})\}$ is equivalent to $\cplx$ (by \Cref{Rmk:Abhy_Sukuki}), and  $u_{j}$'s are harmonic on $\triangle_{k}.$ Hence $u$ is constant on some component of $\triangle_{k},$ i.e. all $h_{j}$'s are holomorphic on that component of $\triangle_{k}.$ We will denote this component again by $\triangle_{k}.$

		We now claim that $\Gr_{h}(\triangle_{k})\subset\widehat{ \Gr_{h}(\{p_1(z)=e^{i\theta_{0}}p_{1}(a_k,b_k),|p_{2}(z)|=r_{j_k}\})}.$ To prove this claim, let $(\xi,h(\xi))\in \Gr_{h}(\triangle_{k})$ for some $\xi\in \triangle_{k}.$ Then for any polynomial $P(z,w)$ in $\cplx^2,$
		\begin{align*}
			|P(\xi,h(\xi))|&\le \sup_{ \Gr_{h}(\triangle_{k})}|P(z,w)|\\
			&=\sup_{\triangle_{k}}|P(z,h(z))|\\
			&=\sup \{|P(z,h(z))|:\;p_1(z)=e^{i\theta_{0}}p_{1}(a_k,b_k), |p_{2}(z)|=r_{j_k}\} .
		\end{align*}
		The last equality holds because $h$ is holomorphic on $\triangle_{k}$.
		Therefore, 
		\[
		(\xi,h(\xi))\in\widehat{ \Gr_{h}(\{p_1(z)=e^{i\theta_{0}}p_{1}(a_k,b_k),|p_{2}(z)|=r_{j_k}\})}.
		\]
		This implies that $\Gr_{h}(\{p_1(z)=e^{i\theta_{0}}p_{1}(a_k,b_k),|p_{2}(z)|=r_{j_k}\})$ is not polynomially convex. Hence, $h$ is holomorphic on some simply connected component $\widetilde{\triangle_{k}}$ of $\{p_1(z)=e^{i\theta_{0}}p_{1}(a_k,b_k),|p_{2}(z)|\le r_{j_k}\}.$ That is, there exists biholomorphism $\phi_{k}:\mathbb{D}\to\widetilde{\triangle_{k}}$ such that $h_{j}\circ\phi_{k}:\mathbb{D}\to \cplx$ is holomorphic  $\forall j=1,\cdots,N;~ \forall k\ge 1.$ Again there exists a subsequence of $(a_k,b_k)$ which is convergent to some point $(a,b).$ Since $|p_{1}(a_k,b_k)|=r_{j_k}$ convergent to $1$ (for the above subsequence), then $p_{1}{(a_k,b_k)}$ converges to $p_{1}(a,b)$ and $|p_{1}(a,b)|=1.$ Hence $\phi_{k}(\mathbb{D})=\widetilde{\triangle_{k}}$ converges to some simply connected component $\widetilde{\triangle}$ of $\{p_1(z)=p_{1}(a,b),|p_{2}(z)|\le 1\}.$ Therefore, there exists a biholomorphism $\phi:\mathbb{D}\to\widetilde{\triangle}$ such that $h_{j}\circ\phi:\mathbb{D}\to \cplx$ is holomorphic  $\forall j=1,\cdots,N.$
	\end{proof}
	
	Without loss of generality, we assume that $u_1$ is non-constant. For $c\in \mathbb{R},$ define $L_{c}=L_{c}(u_1):=\{z\in \mathfrak{D}_{2}:u_{1}(z)=c \}.$ Since $\rl \Phi_{1}=u_{1},$ we have  
	\begin{align*}
		L_{c}(u_1)=\cup_{r\in Img (\Phi_{1})} L_{c+i.r}(\Phi_{1})
	\end{align*}
	is the disjoint union of analytic (since $\Phi_{1}$ is holomorphic) sets, each of them is leaf of $L_c.$
	If $\Phi_{1}(z_{0})=c+ir,$ for some $c+ir\in Img(\Phi_{1}),$ then we denote $\mathfrak{L}_{z_0}:=\{z\in \mathfrak{D}_{2}:\Phi_{1}(z)=\Phi_{1}(z_0)\}$ is the leaf through $z_0.$

	\begin{lemma}\label{L:Lemma4}
		There exists a discrete set $W\subset\mathfrak{D}_{2}$ such that near any point $z_{0}\in \mathfrak{D}_{2}\setminus W,$ the set $\{\Phi_{1}(z)=\Phi_{1}(z_{0})\}$ is a smooth surface.
	\end{lemma}
	\begin{proof}
		Let $Y:= \{\Phi_{1}(z)=\Phi_{1}(z_{0})\}$ be a complex analytic variety of $\mathfrak{D}_2.$ Let $W$ be the collection of singular points of $Y.$ Being a singular set of a complex analytic variety $Y,$ $W$ is also a complex analytic variety with $\dim_{\cplx} W< \dim_{\cplx}Y=1.$ Hence $W$ is discrete. Therefore $Y\setminus W$ form a complex manifold of dimension one (by Implicit function theorem).
	\end{proof}

	\begin{lemma}\label{L:Lemma5}
		Assume that $\Gr_{h}(\alpha_{0})\in \widehat{ \Gr_{h}(\mathfrak{Q}_{j})}\cap \Gr_{h}(\mathfrak{D}^{r_j}_{2})$	and set $c_1=u_{1}(\alpha_{0}).$ Then there exists a point $z_{0}\in (L_{c_1}\setminus W)\cap\mathfrak{D}^{r_j}_{2}$ with $\Gr_{h}(z_{0})\in \widehat{ \Gr_{h}(\mathfrak{Q}_{j})}\cap \Gr_{h}(\mathfrak{D}^{r_j}_{2}).$
	\end{lemma}
	\begin{proof}
		Since $\mathfrak{Q}_{j}$ is a compact subset of $\mathfrak{D}_{2}$ and $\Gr_{h}(\alpha_{0})\in \widehat{ \Gr_{h}(\mathfrak{Q}_{j})}$, by \Cref{L:biholo_cordnte}, $\Gr_{u}(\alpha_{0})\in \widehat{ \Gr_{u}(\mathfrak{Q}_{j})}.$ Let us define
		\begin{align*}
			{\bf c}:=(c_1,\cdots,c_{N})=u(\alpha_{0}), \text{ and }\quad
			\mathfrak{Q}^{{\bf c}}_{j}:=\{z\in\mathfrak{Q}_{j}:u(z)={\bf c}\}.
		\end{align*}
		
		\noindent We consider the map $F:\cplx_{z}^{2}\times \cplx_{w}^{N}\to \cplx^{N}$ defined by 
		\begin{align*}
			F(z,w)=w-{\bf c}.
		\end{align*}
		Let $K:=\Gr_{u}(\mathfrak{Q}_{j}).$ Then, 
		\begin{itemize}
			\item $Y:=F(K)=\{u(z)-{\bf c}: z\in \mathfrak{Q}_{j} \}\subset\mathbb{R}^{N}$;
			\item $0\in Y,$ because $\alpha_{0}\in \mathfrak{Q}^{{\bf c}}_{j}\subset \mathfrak{Q}_{j}$;
			\item $F^{-1}\{0\}\cap K=\Gr_{u}(\mathfrak{Q}^{{\bf c}}_{j})$: let 
			\begin{align*}
				(z,w)\in F^{-1}\{0\}\cap K &\Longleftrightarrow F(z,w)=0 \text{ and } (z,w)\in \Gr_{u}(\mathfrak{Q}_{j})\\
				& \Longleftrightarrow w={\bf c} \text{ and } (w=u(z), z\in \mathfrak{Q}_{j})\\
				& \Longleftrightarrow (z,w)\in \Gr_{u}(\mathfrak{Q}^{{\bf c}}_{j}).
			\end{align*} 
		\end{itemize}
		\noindent Therefore, by \Cref{R:Polycnvx_fiber}, we obtain that
		\begin{align*}
			\widehat{\Gr_{u}(\mathfrak{Q}_{j})}\cap F^{-1}(0)&=\widehat{(\Gr_{u}(\mathfrak{Q}_{j})\cap F^{-1}\{0\}})=\widehat{\Gr_{u}(\mathfrak{Q}^{{\bf c}}_{j})}.
		\end{align*}
		\noindent Since $\Gr_{u}(\alpha_{0})\in \widehat{\Gr_{u}(\mathfrak{Q}_{j})}$ and $\Gr_{u}(\alpha_{0})\in F^{-1}\{0\},$ hence $\Gr_{u}(\alpha_{0})\in \widehat{\Gr_{u}(\mathfrak{Q}^{{\bf c}}_{j})}.$ 
		
		\noindent Now,
		\begin{align*}
			\mathfrak{Q}^{{\bf c}}_{j}\subset \widetilde{L_{{\bf c}}}:=\{z\in \overline{\mathfrak{D}^{r_j}_{2}}:u(z)={\bf c} \} 
		\end{align*} 
		and $\widetilde{L_{{\bf c}}}$ is polynomially convex. Therefore, 
		\begin{align*}
			\widehat{\mathfrak{Q}^{\bf c}_{j}}\subset \widehat{\widetilde{L_{\bf c}}}=\widetilde{L_{\bf c}}\subset L_{c_1}.
		\end{align*}
		\noindent	This implies 
		\begin{align*}
			\widehat{\Gr_{u}(\mathfrak{Q}^{\bf c}_{j})}\subset \widehat{\Gr_{u}(\widehat{\mathfrak{Q}^{\bf c}_{j}})}\subset \widehat{\mathfrak{Q}^{\bf c}_{j}}\times \cplx^{N}\subset L_{c_1}\times\cplx^{N}.
		\end{align*}
		\noindent Hence, $\Gr_{u}(\alpha_{0})\in \widehat{\Gr_{u}(\mathfrak{Q}^{\bf c}_{j})}$ implies
		\begin{align}\label{E:Lc_In_Lc1}
			\alpha_{0}\in \widehat{\mathfrak{Q}^{\bf c}_{j}}\subset\widetilde{ L_{\bf c}}\subset L_{c_1}.
		\end{align}
		If $\alpha_{0}\in L_{c_1}\setminus W,$ we are done. If $\widehat{\mathfrak{Q}^{\bf c}_{j}}$ contains only singular points and since singular points are isolated, $\widehat{\mathfrak{Q}^{\bf c}_{j}}$ is finite. Therefore, $\alpha_{0}\in\widehat{\mathfrak{Q}^{\bf c}_{j}}=\mathfrak{Q}^{\bf c}_{j},$ but we have taken $\alpha_{0}$ from $\mathfrak{D}^{r_j}_{2}.$ This is a contradiction. Therefore, $\widehat{\mathfrak{Q}^{\bf c}_{j}}$ contains at least one point $z_{0}\in L_{c_1}\setminus W.$
		
		\noindent We claim that $\Gr_{u}(z_{0})\in\widehat{ \Gr_{u}(\mathfrak{Q}_{j})}.$ To prove this claim, let $H(z,w)$ be a polynomial in $\cplx^{2+N}.$ Since $(z_0,u(z_0))\in \Gr_{u}(\widehat{\mathfrak{Q}^{\bf c}_{j}})$, then 
		\begin{align*}
			|H(z_0,u(z_0))|&\le\sup_{\Gr_{u}(\widehat{\mathfrak{Q}^{\bf c}_{j}})}|H(z,w)|\\
			&= \sup_{\widehat{\mathfrak{Q}^{\bf c}_{j}}}|H(z,u(z))|\\
			&=\sup_{\widehat{\mathfrak{Q}^{\bf c}_{j}}}|H(z,\bf c)|, \text{ since } \widehat{\mathfrak{Q}^{\bf c}_{j}}\subset \widetilde{ L_{\bf c}}\\
			&\le \sup_{\mathfrak{Q}^{\bf c}_{j}}|H(z,\bf c)|\\
			&=\sup_{\mathfrak{Q}^{\bf c}_{j}}|H(z,u(z))|, \text{ since } \mathfrak{Q}^{\bf c}_{j}\subset \widetilde{ L_{\bf c}}\\
			&\le \sup_{\mathfrak{Q}_{j}}|H(z,u(z))|, \text{ since } \mathfrak{Q}^{\bf c}_{j}\subset\mathfrak{Q}_{j}.
		\end{align*}
		This implies $(z_0,u(z_0))\in \widehat{\Gr_{u}(\mathfrak{Q}_{j})}.$ In view \Cref{L:biholo_cordnte}, we have that $(z,h(z))\in \Gr_{h}(\widehat{\mathfrak{Q}_{j}})$ which completes the proof. 
	\end{proof}

	\begin{lemma}\label{L:Lemma6}
		For any point $z_{0}\in L_{c_1}\setminus W,$ we have that $\Gr_{h}(z_{0})\in \widehat{\Gr_{h}(\mathfrak{Q}_{j})}$ implies that $\Gr_{h}(\mathfrak{L}_{z_0}\cap \mathfrak{D}^{r_j}_{2})\subset \widehat{ \Gr_{h}(\mathfrak{Q}_{j})}$ with $h$ holomorphic along $\mathfrak{L}_{z_0}$ where $\mathfrak{L}_{z_0}$ denotes the leaf through $z_{0}.$
	\end{lemma}
	
	First, we prove the following result.

	\begin{lemma}\label{L:Lemma7}
		Let $z_0\in (L_{c_1}\setminus W)\cap \mathfrak{D}^{r_j}_{2}$ and $\mathfrak{L}_{z_0}$ be the leaf through $z_0.$ Then $z_0\in\widehat{\mathfrak{Q}^{\bf c}_{j}}$ if and only if $\mathfrak{L}_{z_0}\subset \widehat{\mathfrak{Q}^{\bf c}_{j}}.$
	\end{lemma}
	
	\begin{proof}
		If $\mathfrak{L}_{z_0}\subset \widehat{\mathfrak{Q}^{\bf c}_{j}},$ then $z_0\in\widehat{\mathfrak{Q}^{\bf c}_{j}}.$ We need to show only the other side. Let $U$ be a neighborhood of $z_{0}$ and $D(z_{0})=\mathfrak{L}_{z_0}\cap U$ is a smoothly bounded disc containing $z_0.$ Using a translation by $-\Phi_{1}(z_0),$ without loss of generality, we may assume that $\Phi_{1}(z_0)=0 \implies c_1=0.$ For $z\in L_{c_{1}}, u_{1}(z)+iv_1(z)=iv_{1}(z)\in i\mathbb{R}.$ Hence
		\begin{align}\label{E:Lc_minus_leaf}
			\Phi_{1}(L_{c_1})\subset i\mathbb{R} \text{ and } \Phi_{1}(L_{c_1}\setminus \mathfrak{L}_{z_0} )\subset i\mathbb{R}\setminus\{0\}.
		\end{align}
		
		\noindent Let $X:=\widehat{\mathfrak{Q}^{\bf c}_{j}}\cap \partial U.$ By \Cref{R:Rossi}, we get that $z_{0}\in \widehat{X}.$ We claim that $\partial D(z_0)=(\partial U\cap\mathfrak{L}_{z_0})\subset X.$ If not, consider the set $X_{z_0}:=\partial D(z_0)\cap X.$ Then $X_{z_0}$ is a proper subset of $\partial D(z_0).$ This implies $z_{0}\notin \widehat{X_{z_0}}.$ Then there exists a polynomial $\gamma$ such that 
		\begin{align*}
			\gamma(z_0)=1 ~~~\text{ and } \sup_{X_{z_0}}|\gamma|<1.
		\end{align*}
		\noindent Again, consider the function $\xi(t)=e^{t^{2}}$ from $\cplx$ to $\cplx.$ We consider the function
		\begin{align*}
			\phi(z_1,z_2)=\gamma(z_1,z_2)\times(\xi\circ \Phi_{1})^{d}(z_1,z_2),~ d\in \mathbb{N}.
		\end{align*}
		
		\noindent Then $\phi(z_0)=\gamma(z_0)\times(\xi\circ \Phi_{1}(z_0))^{d}=1.$ Now we show that $\sup_{X}|\phi|<1.$ Since $X$ is compact, $\sup_{X}|\phi|=|\phi(s)|$ for some $s\in X.$ Note that 
		\begin{align*}
			s\in X= \widehat{\mathfrak{Q}^{\bf c}_{j}}\cap \partial U&\subset L_{\bf c}\cap \partial U (\text{ by } (\ref{E:Lc_In_Lc1}))\\
			&\subset (L_{\bf c}\cap\bar{U}\setminus D(z_0))\cup (\mathfrak{L}_{z_0}\cap \partial U).
		\end{align*}
		
		\noindent  {\bf Case I} If $s\in \mathfrak{L}_{z_0}\cap \partial U$ and since $s\in X,$ therefore $s\in X_{z_0}=X\cap \partial D(z_0),$ then we have $|\gamma(s)|<1.$ Also, since $\Phi_{1}(s)\in i\mathbb{R},$ there exists $y\in \mathbb{R}$ such that $\Phi_{1}(s)=iy.$ Therefore, $(\xi\circ \Phi_{1}(s))=e^{-y^2}<1.$ This implies $|(\xi\circ \Phi_{1}(s))^{d}|\le 1.$  This contradicts that $z_{0}\in \widehat{X}.$
		
		\medskip
		
		\noindent {\bf Case II} If $s\in (L_{\bf c}\cap\bar{U}\setminus D(z_0))$ and $s\notin \mathfrak{L}_{z_0}\cap \partial U$ implies $s\notin \mathfrak{L}_{z_0},$ by (\ref{E:Lc_minus_leaf}), $|\xi\circ \Phi_{1}(s))|<1.$ Then for sufficiently large $d$, we have $|\gamma(s)|\times|\xi\circ \Phi_{1}(s))|^{d}<1.$ Again this is a contradiction to the fact that $z_{0}\in \widehat{X}.$ Therefore, $X_{z_0}=\partial D(z_0)\subset X.$ Hence, 
		\begin{align}\label{E:Disc_Presnt_Level_Set}
			D(z_0)\subset \widehat{\partial D(z_0)}\subset \widehat{X}\subset \widehat{\mathfrak{Q}^{\bf c}_{j}}\subset \widetilde{ L_{\bf c}}.
		\end{align}
		Proof of \Cref{L:Lemma7} follows as $\mathfrak{L}_{z_0}$ can be covered by such $D(z_0).$
	\end{proof}
	
	We now complete the proof of \Cref{L:Lemma6}.
	
	\begin{proof}[Proof of \Cref{L:Lemma6}]
		let $z_{0}$ be as in the hypothesis. Then by \Cref{L:Lemma5}, $z_{0}\in \widehat{\mathfrak{Q}^{\bf c}_{j}}.$ By \Cref{L:Lemma7}, $\mathfrak{L}_{z_0}\subset \widehat{\mathfrak{Q}^{\bf c}_{j}}.$ For $z\in \mathfrak{L}_{z_0}\cap \mathfrak{D}^{r_j}_{2},$ and for any polynomial $Q(z,w)$ we have:
		\begin{align*}
			|Q(z,u(z))|&\le\sup_{\Gr_{u}(\widehat{\mathfrak{Q}^{\bf c}_{j}})}|Q(z,w)|\\
			&=\sup_{\widehat{\mathfrak{Q}^{\bf c}_{j}}}|Q(z,{\bf c})|, \text{ since } \widehat{\mathfrak{Q}^{\bf c}_{j}}\subset \widetilde{L_{\bf c}}\\
			&= \sup_{\mathfrak{Q}^{\bf c}_{j}}|Q(z,{\bf c})|\\
			&=\sup_{\mathfrak{Q}^{\bf c}_{j}}|Q(z,u(z))|, \text{ since } \mathfrak{Q}^{\bf c}_{j}\subset \widetilde{L_{\bf c}}\\
			&\le\sup_{\mathfrak{Q}_{j}}|Q(z,u(z))|.
		\end{align*}
		Therefore, $(z,u(z))\in \Gr_{u}(\widehat{\mathfrak{Q}_{j}}).$ Hence $(z,h(z))\in \Gr_{h}(\widehat{\mathfrak{Q}_{j}}).$  From (\ref{E:Disc_Presnt_Level_Set}), we get that $u_2,\cdots,u_{N}$ are also constant on $D(z_0).$ Hence $h_{j}$'s are holomorphic on $D(z_0)=U\cap\mathfrak{L}_{z_0}$ which completes the proof of \Cref{L:Lemma6}.
	\end{proof}
	
	\medskip

	\begin{proof}[ Proof of \Cref{T:PPolyhrn_DistingBdry}]
		If there exists an analytic set $\triangle\subset\left(\partial{\mathfrak{D}_{2}}-\Gamma_{\mathfrak{D}_{2}}\right),$ where all $h_{j}$'s are holomorphic, then there is nothing to prove. We assume that no such analytic set exists where all $h_{j}$'s are holomorphic. Therefore, by \Cref{L:Lemma1}, we have $\widehat{\Gr_{h}(\Gamma_{\mathfrak{D}_{2}})}\cap \Gr_{h}\left(\partial{\mathfrak{D}_{2}}-\Gamma_{\mathfrak{D}_{2}}\right)=\emptyset.$ Recall the definition of $\widetilde{V}$ (see \ref{E:Set_Graph_Complx_Tngnt}):
		\begin{align*}
			\widetilde{V}:=\{z\in \mathfrak{D}_{2}:\bar{\partial}{{h_{j_1}}}\wedge\bar{\partial}{{h_{j_2}}}(z)=0, \forall~ {j_1,j_2}, 1\le j_1,j_2\le N\}.
		\end{align*}	
		\noindent Since, by assumption in the beginning of the proof,  $\Gr_{h}(\Gamma_{\mathfrak{D}_{2}})$ is not polynomially convex, by (\ref{E:Hull_disbndry}), there exist $\alpha_{0}\in\widetilde{V}$ such that $\Gr_{h}(\alpha_{0})\in\widehat{ \Gr_{h}(\Gamma_{\mathfrak{D}_{2}})}.$ 
		
		\medskip
		
		\noindent {\bf Case I:}
		Assume that $\widetilde{V}\not =\mathfrak{D}_{2}.$ Then \Cref{L:Lemma2} gives an irreducible component $V$ of $\widetilde{V}$ such that all $h_{j}$'s are holomorphic on $V.$ Observe that $\overline{V}\setminus V\subset \Gamma_{\mathfrak{D}_{2}}.$ The proof of the theorem is complete in this case in view of \Cref{L:Tornehave_PP}.
		
		\medskip
		
		\noindent{\bf Case II:} Now we consider the case that $\widetilde{V} =\mathfrak{D}_{2}.$ Without loss of generality, we assume that $h_1$ is non-holomorphic. Since $\Gr_{h}(\alpha_{0})\in\widehat{ \Gr_{h}(\Gamma_{\mathfrak{D}_{2}})},$ there exists $j_{0}\in \mathbb{N}$ such that $\alpha_{0}\in \mathfrak{D}^{r_j}_{2}$ for $j>j_{0}.$ Therefore, by (\ref{E:hull_Qj}) and \Cref{R:Rossi}, we obtain that $\Gr_{h}(\alpha_{0})\in \widehat{\Gr_{h}(\mathfrak{Q}_{j})}:$ in fact, assume that
		\begin{align}\label{E:K in Hull_Grph_DisBdry}
			K:=(\partial \mathfrak{D}^{r_j}_{2}\times\cplx^{N})\cap \widehat{ \Gr_{h}(\Gamma_{\mathfrak{D}_{2}})}.
		\end{align}
		\noindent Since $\widehat{\Gr_{h}(\Gamma_{\mathfrak{D}_{2}})}\subset \Gr_{h}(\Gamma_{\mathfrak{D}_{2}}\cup\mathfrak{D}_{2}),$ therefore, 
		\begin{align}\label{E:K lies_Bdry}
			K\subset \Gr_{h}(\Gamma_{\mathfrak{D}_{2}}\cup\mathfrak{D}_{2})\cap (\partial \mathfrak{D}^{r_j}_{2}\times\cplx^{N})=\Gr_{h}(\partial \mathfrak{D}^{r_j}_{2}).
		\end{align}
		Next we show that $\Gr_{h}(\alpha_{0})\in\widehat{K}.$ To show this we assume that $E=\Gr_{h}(\Gamma_{\mathfrak{D}_{2}})$ and $N=\mathfrak{D}^{r_j}_{2}\times\mathbb{B}(r)$ with $r>\sup_{\mathfrak{D}^{r_j}_{2}}|h|.$ Note that $E\cap N=\emptyset$ and $\partial N\cap\widehat{E}=(\partial\mathfrak{D}^{r_j}_{2}\times\mathbb{B}(r))\cap\widehat{E}\cup (\mathfrak{D}^{r_j}_{2}\times \partial\mathbb{B}(r))\cap\widehat{E}.$ Note that $(\mathfrak{D}^{r_j}_{2}\times \partial\mathbb{B}(r))\cap\widehat{E}=\emptyset.$ Then, by \Cref{R:Rossi}, we have
		\begin{align*}
			\Gr_{h}(\alpha_{0})\in \widehat{\widehat{E}\cap \partial N}&=\widehat{[\widehat{\Gr_{h}(\Gamma_{\mathfrak{D}_{2}})}\cap (\partial\mathfrak{D}^{r_j}_{2}\times\mathbb{B}(r))]}\\
			&\subset \widehat{[\widehat{\Gr_{h}(\Gamma_{\mathfrak{D}_{2}})}\cap (\partial\mathfrak{D}^{r_j}_{2}\times\cplx^{N})]}=\widehat{K}.
		\end{align*}
		\noindent Using (\ref{E:K in Hull_Grph_DisBdry}), (\ref{E:K lies_Bdry}) and(\ref{E:hull_Qj}), we obtain that
		\begin{align*}
			K\subset	\widehat{\Gr_{h}(\Gamma_{\mathfrak{D}_{2}})}\cap \Gr_{h}(\partial\mathfrak{D}^{r_j}_{2})\subset \Gr_{h}(\mathfrak{Q}_{j}).
		\end{align*}	
		
		\noindent Therefore,  $\Gr_{h}(\alpha_{0})\in \widehat{K}$ implies 
		\begin{align*}
			\Gr_{h}(\alpha_{0})\in  \widehat{\Gr_{h}(\mathfrak{Q}_{j})}.
		\end{align*}
		Therefore, we have $\Gr_{h}(\alpha_{0})\in \widehat{\Gr_{h}(\mathfrak{Q}_{j})}\cap \Gr_{h}(\mathfrak{D}^{r_j}_{2}).$ By \Cref{L:Lemma5}, we may assume that $\alpha_{0}\notin W.$ Let $V^{'}$ denotes the irreducible component of $\{z\in \mathfrak{D}_{2}: \Phi_{1}(z)=\Phi_{1}(\alpha_{0}) \}$ containing $\alpha_{0}.$ Then from \Cref{L:Lemma6}, we get that 
		\begin{align}\label{E:Graph_Varity_In_hull}
			\Gr_{h}(V^{'}\cap \mathfrak{D}^{r_j}_{2})\subset \widehat{ \Gr_{h}(\mathfrak{Q}_{j})} \text{ for all } j>j_0.
		\end{align}
		We now claim that $\Gr_{h}(V^{'})\subset \widehat{\Gr_{h}(\Gamma_{\mathfrak{D}_{2}})}.$ To prove this claim, assume that $(\alpha,h(\alpha))\in \Gr_{h}(V^{'}).$ Since $\alpha\in V^{'},$ $\alpha\in V^{'}\cap \mathfrak{D}^{r_{j_1}}_{2}$ for some $j_{1}.$ Therefore, $\alpha\in V^{'}\cap \mathfrak{D}^{r_{j}}_{2}$ for all $j>j_{1}.$ Using (\ref{E:Graph_Varity_In_hull}), for any polynomial $H$ in $\cplx^{N+2}$ and for $j>\max \{j_{0},j_{1}\},$ we get that
		
		\begin{align}\label{E:FromQj to Distinguishd bdry}
			|H(\alpha,h(\alpha))|\le \sup_{\mathfrak{Q}_{j}}|H(z,h(z))|=|H(a_{j},h(a_{j}))| 
		\end{align}
		for some $a_{j}\in \mathfrak{Q}_{j}.$ Taking limit $j\to \infty,$ we obtain that $|H(a_{j},h(a_{j}))|\to H(a,h(a))$ for some $a\in \Gamma_{\mathfrak{D}_{2}}.$ Hence from (\ref{E:FromQj to Distinguishd bdry}), we get that
		\begin{align*}
			|H(\alpha,h(\alpha))|\le \lim_{j\to \infty}|H(a_{j},h(a_{j}))| =|H(a,h(a))|\le \sup_{\Gamma_{\mathfrak{D}_{2}}} |H(z,h(z))|.
		\end{align*}
		Since $H$ is arbitrary, then we can say that $(\alpha,h(\alpha))\in \widehat{\Gr_{h}(\Gamma_{\mathfrak{D}_{2}})}.$ This proves our claim. Note that $\overline{V^{'}}\setminus V^{'}\subset \Gamma_{\mathfrak{D}_{2}}.$ The proof of the theorem is complete in this case in view of \Cref{L:Tornehave_PP}.
	\end{proof}

	\section{Description of hull}\label{sec-polyhull}	
	
	We consider the proper polynomial map $\Psi:\cplx^2\to \cplx^2$ defined by $\Psi(z)=$ $(p_1(z),p_{2}(z)).$ Then the distinguished boundary $\Gamma_{\mathfrak{D}_{2}}$ of the polynomial polyhedron $\mathfrak{D}_{2}=\Psi^{-1}(\mathbb{D}^2)=
	\{z\in\cplx^2:|p_{1}(z)|\le 1,|p_{2}(z)|\le1\}$ is $\Gamma_{{\mathfrak{D}_{2}}}=\{z\in\cplx^2:|p_{1}(z)|= 1,|p_{2}(z)|=1\}.$ Let $P(z_1,z_2)$ be any holomorphic polynomial of degree $m$ in $z_1$ and of degree $n$ in $z_2.$ We wish to compute $\widehat{\Gr_{\overline{P\circ \Psi}}(\Gamma_{\mathfrak{D}_{2}})}$ under the assumption that $\mathfrak{D}_{2}$ is complex non-degenerate.
	
	\smallskip
	
	\par First, we state and prove a lemma that will be useful in this work. For a subset $X$ of $\cplx^{n},$ we denote 
	\begin{align*}
		\# CC(X)= \text{the number of connected components of } X.	
	\end{align*}
	\begin{lemma}\label{L:Nbr_Compnt_Hull}
		Let $X$ be a compact subset of $\cplx^{n}.$ Then $\# CC(\widehat{X})\le \# CC(X).$	
	\end{lemma}
	\begin{proof}
		If $\# CC(X)=\infty,$ then there is nothing to prove. In view \Cref{R:Component_Non_Empty}, it is not possible that $\# CC(X)<\infty,$ but $\# CC(\widehat{X})=\infty.$ In fact, assume that $\widehat{X}=\cup^{\infty}_{j=1} X_{j},$ where each $X_{j}$ is disjoint connected component of $\widehat{X}.$ This implies $X=X\cap \widehat{X}=\cup^{\infty}_{j=1} X_{j}\cap X.$ \Cref{R:Component_Non_Empty} says that $X_{j}=\widehat{X_{j}\cap X}~\forall j\ge 1.$ Since each $X_{j}\cap X$ is disjoint, and $\# CC(X)<\infty,$ therefore there exists $j_{0}\in \mathbb{N}$ such that $X_{j}\cap X=\emptyset ~\forall j>j_{0}.$ Hence $X_{j}=\widehat{X_{j}\cap X}=\emptyset~\forall j>j_{0}.$ This is a contradiction to the fact that $\# CC(\widehat{X})=\infty.$
		
		\medskip
		
		\noindent Next, we assume that $\# CC(X)=m$ and $\# CC(\widehat{X})=l.$ Assume that $\widehat{X}=\cup^{l}_{j=1}X_{j},$  where each $X_{j}$ is disjoint connected component of $\widehat{X}.$ Then $X=\widehat{X}\cap X=\cup^{l}_{j=1}X_{j}\cap X.$ By using \Cref{R:Component_Non_Empty}, we obtain that $X_{j}=\widehat{X_{j}\cap X}~\forall j\ge 1.$ If each $X_{j}\cap X$ is connected, then $l=m.$ Otherwise $l< m.$
	\end{proof}

	\smallskip	
	
	We write
	\begin{align*}
		P(z_1,z_2)=	\sum_{\substack{0\le i\le m\\
				0\le j\le n	\\
		}}
		a_{ij}z_1^{i}z_2^{j}.
	\end{align*}
	
	\noindent Therefore, on $\Gamma_{\mathfrak{D}_{2}},$ we have
	\begin{align*}
		\overline{(P\circ \Psi)(z_1,z_2)}&=\overline{P(p_1,p_2)}
		=	\sum_{\substack{0\le i\le m\\
				0\le j\le n	\\
		}}
		\overline{a}_{ij}\overline{p_{1}}^{i}\overline{{p_2}}^{j}
		=\frac{1}{p^{m}_{1}p^{n}_{2}}	\sum_{\substack{0\le i\le m\\
				0\le j\le n	\\
		}}
		\overline{a}_{ij}p_{1}^{m-i}{p_2}^{n-j}\\
		&=\frac{(K\circ \Psi)(z_1,z_2)}{p^{m}_{1}p^{n}_{2}}=(h\circ\Psi)(z_1,z_2),
	\end{align*}
	\noindent where
	\begin{align*}
		K(z_1,z_2)=\sum_{\substack{0\le i\le m\\
				0\le j\le n	\\
		}}
		\overline{a}_{ij}z_1^{m-i}z_2^{n-j} \text{ and }\quad 
		h(z_1,z_2)=\frac{K(z_1,z_2)}{z^{m}_1z^{n}_{2}}.
	\end{align*}
	
	\noindent We define
	\begin{align*}
		L:=&\{p_1=0,|p_2|\le 1\}\cup\{p_2=0,|p_1|\le 1\} \text{ and }\\
		X:=&\left\{(z_1,z_2)\in \overline{\mathfrak{D}_{2}}\setminus (L\cup \Gamma_{\mathfrak{D}_{2}}): \overline{(P\circ \Psi)(z_1,z_2)}=(h\circ\Psi)(z_1,z_2)\right\}.
	\end{align*}

	\begin{lemma}\label{L:Tot_real Part}
		$X$ is totally real at $z_{0}\in X,$ if 
		\begin{align*}
			\triangle({z_0}):=
			\begin{vmatrix} 
				\frac{\partial P(\Psi)}{\partial{z_1}}& \frac{\partial P(\Psi)}{\partial z_2}\\[1.5ex]
				\frac{\partial h(\Psi)}{\partial{{z_1}}}& \frac{\partial h(\Psi)}{\partial{z_2}}\\[1.5ex]
			\end{vmatrix}_{z_0}
			\begin{vmatrix} 
				\frac{\partial p_1}{\partial{{z_1}}}& \frac{\partial p_1}{\partial{z_2}}\\[1.5ex]
				\frac{\partial p_2}{\partial{{z_1}}}& \frac{\partial p_2}{\partial{z_2}}\\[1.5ex]
			\end{vmatrix}_{z_0}\ne 0.
		\end{align*}	
	\end{lemma}
	\begin{proof}
		Take
		\begin{align*}
			\rho_1&=\rl\left(\overline{(P\circ \Psi)(z_1,z_2)}-(h\circ\Psi)(z_1,z_2)\right),\\
			\rho_2&=\imag\left(\overline{(P\circ \Psi)(z_1,z_2)}-(h\circ\Psi)(z_1,z_2)\right).
		\end{align*}
		
		\noindent	Therefore,
		\begin{align*}
			\rho_1(z)&=\frac{1}{2}\left(\overline{(P\circ \Psi)(z)}-(h\circ\Psi)(z)+(P\circ \Psi)(z)-\overline{(h\circ\Psi)(z)}\right)\\
			\rho_2(z)&=\frac{1}{2i}\left(\overline{(P\circ \Psi)(z)}-(h\circ\Psi)(z)-(P\circ \Psi)(z)+\overline{(h\circ\Psi)(z)}\right).
		\end{align*}
		
		\noindent	Since 	
		\begin{align*}
			X=\left\{(z,w)\in \overline{\mathfrak{D}_{2}}\setminus (L\cup \Gamma_{\mathfrak{D}_{2}}): \rho_1=0,\rho_2=0\right\},
		\end{align*}
		$X$ is totally real at $z_{0}\in X$ if 
		$$
		\begin{vmatrix} 
			\frac{\partial \rho_1}{\partial{{z_1}}}& \frac{\partial \rho_1}{\partial{z_2}}\\[1.5ex]
			\frac{\partial \rho_2}{\partial{{z_1}}}& \frac{\partial \rho_2}{\partial{z_2}}\\[1.5ex]
		\end{vmatrix}_{z_{0}}\ne 0.$$
		
		\noindent We compute:
		\begin{align*}
			\frac{\partial \rho_1}{\partial{z_j}}&=\frac{1}{2}\left(-\frac{\partial (h\circ \Psi)}{\partial{z_j}}+\frac{\partial (P\circ \Psi)}{\partial{z_j}}\right)\\
			\frac{\partial \rho_2}{\partial{z_j}}&=\frac{1}{2i}\left(-\frac{\partial (h\circ \Psi)}{\partial{z_j}}-\frac{\partial (P\circ \Psi)}{\partial{z_j}}\right),
		\end{align*}
		for $j=1,2.$ And
		\begin{align*}
			\frac{\partial (h\circ\Psi)}{\partial{z_j}}&=\frac{\partial h( \Psi)}{\partial{z_1}}\frac{\partial p_1}{\partial{z_j}}+\frac{\partial h( \Psi)}{\partial{z_2}}\frac{\partial p_2}{\partial{z_j}}\\
			\frac{\partial (P\circ\Psi)}{\partial{z_j}}&=\frac{\partial P( \Psi)}{\partial{z_1}}\frac{\partial p_1}{\partial{z_j}}+\frac{\partial P( \Psi)}{\partial{z_2}}\frac{\partial p_2}{\partial{z_j}},~ j=1,2.
		\end{align*}
		\noindent Hence 
		\begin{align*}
			\begin{vmatrix} 
				\frac{\partial \rho_1}{\partial{{z_1}}}& \frac{\partial \rho_1}{\partial{z_2}}\\[1.5ex]
				\frac{\partial \rho_2}{\partial{{z_1}}}& \frac{\partial \rho_2}{\partial{z_2}}\\[1.5ex]
			\end{vmatrix}_{z_{0}}
			=&\begin{vmatrix}
				\frac{1}{2}\left(-\frac{\partial (h\circ \Psi)}{\partial{z_1}}+\frac{\partial (P\circ \Psi)}{\partial{z_1}}\right)& \frac{1}{2}\left(-\frac{\partial (h\circ \Psi)}{\partial{z_2}}+\frac{\partial (P\circ \Psi)}{\partial{z_2}}\right)\\[1.5ex]
				\frac{1}{2i}\left(-\frac{\partial (h\circ \Psi)}{\partial{z_1}}+\frac{\partial (P\circ \Psi)}{\partial{z_1}}\right)& \frac{1}{2i}\left(-\frac{\partial (h\circ \Psi)}{\partial{z_2}}+\frac{\partial (P\circ \Psi)}{\partial{z_2}}\right)\\[1.5ex]
			\end{vmatrix}_{z_{0}}\\
			=&\frac{1}{2i}\begin{vmatrix}
				-\frac{\partial (h\circ \Psi)}{\partial{z_1}}&-\frac{\partial (h\circ \Psi)}{\partial{z_2}}\\[1.5ex]
				\frac{\partial (P\circ \Psi)}{\partial{z_1}}&\frac{\partial (P\circ \Psi)}{\partial{z_2}}\\[1.5ex]
			\end{vmatrix}_{z_{0}}.
		\end{align*}

		\noindent Therefore,
		\begin{align*}
			\begin{vmatrix} 
				\frac{\partial \rho_1}{\partial{{z_1}}}& \frac{\partial \rho_1}{\partial{z_2}}\\[1.5ex]
				\frac{\partial \rho_2}{\partial{{z_1}}}& \frac{\partial \rho_2}{\partial{z_2}}\\[1.5ex]
			\end{vmatrix}_{z_{0}}\ne 0 \text{ if and only if }
			&\begin{vmatrix}
				\frac{\partial (P\circ \Psi)}{\partial{z_1}}&\frac{\partial (P\circ \Psi)}{\partial{z_2}}\\[1.5ex]
				\frac{\partial (h\circ \Psi)}{\partial{z_1}}&\frac{\partial (h\circ \Psi)}{\partial{z_2}}\\[1.5ex]
			\end{vmatrix}_{z_{0}}
			&=\begin{vmatrix} 
				\frac{\partial P(\Psi)}{\partial{z_1}}& \frac{\partial P(\Psi)}{\partial z_2}\\[1.5ex]
				\frac{\partial h(\Psi)}{\partial{{z_1}}}& \frac{\partial h(\Psi)}{\partial{z_2}}\\[1.5ex]
			\end{vmatrix}_{z_{0}}
			\times
			\begin{vmatrix} 
				\frac{\partial p_1}{\partial{{z_1}}}& \frac{\partial p_1}{\partial{z_2}}\\[1.5ex]
				\frac{\partial p_2}{\partial{{z_1}}}& \frac{\partial p_2}{\partial{z_2}}\\[1.5ex]
			\end{vmatrix}_{z_0}\ne 0.	
		\end{align*}
	\end{proof}

	We can write 
	\begin{align*}
		\triangle(z)=
		\begin{vmatrix} 
			\frac{\partial P(\Psi)}{\partial{z_1}}& \frac{\partial P(\Psi)}{\partial z_2}\\[1.5ex]
			\frac{\partial h(\Psi)}{\partial{{z_1}}}& \frac{\partial h(\Psi)}{\partial{z_2}}\\[1.5ex]
		\end{vmatrix}
		\times
		\begin{vmatrix} 
			\frac{\partial p_1}{\partial{{z_1}}}& \frac{\partial p_1}{\partial{z_2}}\\[1.5ex]
			\frac{\partial p_2}{\partial{{z_1}}}& \frac{\partial p_2}{\partial{z_2}}\\[1.5ex]
		\end{vmatrix}_{z}=\frac{1}{p^{m+1}_{1}p^{n+1}_{2}}\Pi^{l}_{j=1}q_{j}(z),
	\end{align*}
	for some $l\in \mathbb{N},$ where each $q_{j}$ is an irreducible polynomial in $\cplx^2.$ We define the corresponding irreducible algebraic variety $Z_{j}:=\{z\in \cplx^2:q_{j}(z)=0\}.$ We assume $\triangle({z})\not\equiv 0$ on $X.$ Therefore, each $q_{j}$ is a non-zero holomorphic polynomial in $\cplx^2.$

	\noindent We set 
	\begin{align}\label{Set:Complex Pt}
		\Sigma:=\left\{z\in \overline{\mathfrak{D}_{2}}\setminus (L\cup \Gamma_{\mathfrak{D}_{2}}): \triangle{(z)}=0\right\}.
	\end{align}
	
	\noindent We denote 
	\begin{align}\label{E:Defn_Qj}
		Q_{j}=Z_{j}\cap \Gamma_{\mathfrak{D}_{2}}.
	\end{align}
	\begin{theorem}\label{T:Analogus_Jimbo}
		We let $J=\{j\in \{1,\cdots,l\}: \emptyset\ne Q_{j}\ne \widehat{Q_{j}}, \widehat{Q_{j}}\setminus L\subset X\}.$ 
		\begin{enumerate}
			\item [(i)]If $J=\emptyset,$ then $\widehat{\Gr_{\overline{P\circ \Psi}}(\Gamma_{\mathfrak{D}_{2}})}=\Gr_{\overline{P\circ \Psi}}(\Gamma_{\mathfrak{D}_{2}}).$
			\item[(ii)] If $J\ne \emptyset,$ then 
		\end{enumerate}

		\begin{align*}
			\widehat{\Gr_{\overline{P\circ \Psi}}(\Gamma_{\mathfrak{D}_{2}})}=\Gr_{\overline{P\circ \Psi}}(\Gamma_{\mathfrak{D}_{2}})\cup\bigg(\cup_{j\in J}  \Gr_{\overline{P\circ \Psi}}(\widehat{Q_{j}})\bigg).
		\end{align*}
	\end{theorem}
	\begin{remark}
		If (i) of \Cref{T:Analogus_Jimbo} is true, then by \Cref{T: Approx_Cont_Func}, we get that
		\begin{align*}
			[z_1,z_2,\overline{P\circ\Psi}]_{\Gamma_{{\mathfrak{D}_{2}}}}=\smoo(\Gamma_{{\mathfrak{D}_{2}}}).
		\end{align*}
	\end{remark}
	
	We will prove this theorem through some lemmas.
	\begin{lemma}\label{L:Poly_Cnvx_Hull}
		$\widehat{\Gr_{\overline{P\circ \Psi}}(\Gamma_{\overline{\mathfrak{D}_{2}}})}\subset
		\Gr_{\overline{P\circ \Psi}}(X\cap \Sigma) \cup  \Gr_{\overline{P\circ \Psi}}(L\cup \Gamma_{\mathfrak{D}_{2}}),$ where $\Sigma$ as in (\ref{Set:Complex Pt}).
	\end{lemma}
	\begin{proof}
		First, we show that 
		
		\begin{align*}
			\widehat{\Gr_{\overline{P\circ \Psi}}(\Gamma_{\mathfrak{D}_{2}})}\subset
			\Gr_{\overline{P\circ \Psi}}(X) \cup  \Gr_{\overline{P\circ \Psi}}(L\cup \Gamma_{\mathfrak{D}_{2}}). 
		\end{align*}
		Since $\overline{\mathfrak{D}_{2}}$ is polynomially convex, $\mathfrak{D}_{2}$ has the segment property, and ${\overline{P\circ \Psi}}$ is pluriharmonic function, therefore using \Cref{L:Poly_Cnvx}, we can say that
		
		\begin{align}\label{Set:Hull 1}
			\widehat{\Gr_{\overline{P\circ \Psi}}(\Gamma_{\mathfrak{D}_{2}})}\subset
			\Gr_{\overline{P\circ \Psi}}(\overline{\mathfrak{D}_{2}}).
		\end{align}
		
		\noindent We now compute:
		\begin{align*}
			\notag 	\Gr_{\overline{P\circ \Psi}}(\Gamma_{\mathfrak{D}_{2}})&=\left\{(z_1,z_2,\xi):\xi=\overline{(P\circ \Psi)(z_1,z_2)}, (z_1,z_2)\in \Gamma_{\mathfrak{D}_{2}} \right\}\\\notag
			&=\left\{(z_1,z_2,\xi):\xi= \overline{(P\circ \Psi)(z_1,z_2)}=(h\circ\Psi)(z_1,z_2), (z_1,z_2)\in \Gamma_{\mathfrak{D}_{2}} \right\}\\\notag
			&=\left\{(z_1,z_2,\xi):\xi=(h\circ\Psi)(z_1,z_2)=\frac{(K\circ\Psi)(z_1,z_2)}{p^{m}_1p^{n}_2}, (z_1,z_2)\in \Gamma_{\mathfrak{D}_{2}} \right\}\\
			&=\left\{(z_1,z_2,\xi):\xi{p^{m}_1p^{n}_2}-{(K\circ\Psi)(z_1,z_2)}=0, (z_1,z_2)\in \Gamma_{\mathfrak{D}_{2}} \right\}.
		\end{align*}
		Therefore,
		\begin{align}\label{Set:Hull 2}
			\widehat{\Gr_{\overline{P\circ \Psi}}(\Gamma_{\mathfrak{D}_{2}})}\subset\left\{(z_1,z_2,\xi):\xi{p^{m}_1p^{n}_2}-{(K\circ\Psi)(z_1,z_2)}=0, (z_1,z_2)\in \overline{\mathfrak{D}}_{2} \right\}.	
		\end{align}
		
		\noindent Using (\ref{Set:Hull 1}) and (\ref{Set:Hull 2}), we get that 
		\begin{align*}
			\widehat{\Gr_{\overline{P\circ \Psi}}(\Gamma_{\mathfrak{D}_{2}})}\subset
			\Gr_{\overline{P\circ \Psi}}(X) \cup  \Gr_{\overline{P\circ \Psi}}(L\cup \Gamma_{\mathfrak{D}_{2}}). 
		\end{align*}
		
		\noindent Again, since by \Cref{L:Tot_real Part}, $X\setminus\Sigma$ is totally real, using \Cref{L:Tot_rlpt_Pk pt}, we obtain that 
		\begin{align*}
			\widehat{\Gr_{\overline{P\circ \Psi}}(\Gamma_{\mathfrak{D}_{2}})}\subset\Gr_{\overline{P\circ \Psi}}(X\cap \Sigma) \cup  \Gr_{\overline{P\circ \Psi}}(L\cup \Gamma_{\mathfrak{D}_{2}}),	
		\end{align*}		
		which completes the proof.	
	\end{proof}
	
	By \Cref{L:Nbr_Compnt_Hull}, we have that the number of connected component of $\widehat{\Gr_{\overline{P\circ \Psi}}(\Gamma_{\mathfrak{D}_{2}})}$ is less than or equal to the number of connected components of $\Gr_{h}(\Gamma_{\mathfrak{D}_{2}}).$ Therefore, it is enough to consider the set $\Gr_{\overline{P\circ \Psi}}(X\cap \Sigma).$

	\begin{lemma}\label{L:Hull_RegularSet}
		Fix $j_{0}\in \{1,\cdots,l\}.$	If there exists $\alpha\in (Z_{j_0})_{reg}$ such that $\Gr_{\overline{P\circ \Psi}}(\alpha)\in\widehat{\Gr_{\overline{P\circ \Psi}}(\Gamma_{{\mathfrak{D}_{2}}})},$ then $P\circ\Psi=c_{j_0}, \text{ on } Z_{j_0}$ for some constant $c_{j_0}\in \cplx.$
	\end{lemma}
	\begin{proof}
		We choose a neighborhood $U$ of $\alpha$ such that $U\cap Z_{j_0}=:D_{\alpha}$ is a smoothly bounded disc containing $\alpha.$ Let $X:=\widehat{\Gr_{\overline{(P\circ \Psi)}}(\Gamma_{\mathfrak{D}_{2}})}\cap(\partial U\times\cplx^{N}).$ Since $\widehat{\Gr_{\overline{(P\circ \Psi)}}(\Gamma_{\mathfrak{D}_{2}})}\subset \Gr_{\overline{(P\circ \Psi)}}(\Gamma_{\mathfrak{D}_{2}}\cup [\cup^{l}_{j=1}Z_{j} ]),$ we have 
		\begin{align*}
			X\subset \Gr_{\overline{(P\circ \Psi)}}(\Gamma_{\mathfrak{D}_{2}}\cup [\cup^{l}_{j=1}Z_{j} ])\cap (\partial U\times\cplx^{N})\subset \Gr_{\overline{(P\circ \Psi)}}(\partial U\cap Z_{j_0}).
		\end{align*}

		\noindent We claim that		
		\begin{align*}
			\Gr_{\overline{(P\circ \Psi)}}(\alpha)\in \widehat{X}\subset \widehat{\Gr_{\overline{(P\circ \Psi)}}(\partial U\cap V)}= \widehat{\Gr_{\overline{(P\circ \Psi)}}(\partial D_{\alpha})}.
		\end{align*}
		
		\noindent To prove this claim, let $E=\Gr_{h}(\Gamma_{\mathfrak{D}_{2}})$ and $N=U\times\mathbb{B}(r)$ with $r>\sup_{U}|\overline{(P\circ \Psi)}|.$ Note that $E\cap N=\emptyset$ and $\partial N\cap\widehat{E}=[(\partial U\times\mathbb{B}(r))\cap\widehat{E}]\cup [(U\times \partial\mathbb{B}(r))\cap\widehat{E}].$ We claim that $(U\times \partial\mathbb{B}(r))\cap\widehat{E}=\emptyset:$ if possible, let $(\eta,\xi)\in (U\times \partial\mathbb{B}(r))\cap\widehat{E}.$ Then $\xi=\overline{(P\circ \Psi)}(\eta)$ and $\eta\in U, |\xi|=r.$ But on $U,$ we have $|\xi|=|\overline{(P\circ \Psi)}(\eta)|<r.$ This is a contradiction to the assumption that $(U\times \partial\mathbb{B}(r))\cap\widehat{E}\not =\emptyset.$ Then by \Cref{R:Rossi}, we have
		\begin{align*}
			\Gr_{\overline{(P\circ \Psi)}}(\alpha)\in \widehat{\widehat{E}\cap \partial N}&=\widehat{[\widehat{\Gr_{\overline{(P\circ \Psi)}}(\Gamma_{\mathfrak{D}_{2}})}\cap (\partial U\times\mathbb{B}(r))]}\\
			&\subset \widehat{[\widehat{\Gr_{\overline{(P\circ \Psi)}}(\Gamma_{\mathfrak{D}_{2}})}\cap (\partial U\times\cplx^{N})]}\\
			&=\widehat{X}\subset \widehat{\Gr_{\overline{(P\circ \Psi)}}(\partial U\cap V)}= \widehat{\Gr_{\overline{(P\circ \Psi)}}(\partial D_{\alpha})}.
		\end{align*}
		
		\noindent Therefore, $\widehat{\Gr_{\overline{(P\circ \Psi)}}(\partial D_{\alpha})}$ is not polynomially convex, hence by the Wermer's maximality theorem we can say that $\overline{(P\circ \Psi)}$ is holomorphic in $D_{\alpha}.$ If $\alpha$ runs over regular points of $Z_{j},$ the corresponding neighborhoods $D_{\alpha}$ cover $(Z_{j_0})_{\text{reg}}.$	
	\end{proof}

	\begin{lemma}\label{L:Z_jNot_InHull}
		Fix $j_{0}\in \{1,\cdots,l\}.$ If $P\circ\Psi\equiv c_{j_0}, \text{ on } Z_{j_0}$ for some constant $c_{j_0}\in \cplx$ but $h\circ\Psi\not \equiv \bar{c}_{j_0}$ on $Z_{j_0}\cap \overline{\mathfrak{D}}_{2}\setminus L,$ then 
		\begin{align*}
			\widehat{\Gr_{\overline{P\circ \Psi}}(\Gamma_{\mathfrak{D}_{2}})}\cap \Gr_{\overline{P\circ \Psi}}\left(X\cap Z_{j_0}\cap \overline{\mathfrak{D}}_2\setminus (L\cup \Gamma_{{\mathfrak{D}_{2}}})\right)=\emptyset.
		\end{align*}
	\end{lemma}
	\begin{proof}
		Since $P\circ\Psi- c_{j_0}\equiv 0, \text{ on } Z_{j_0}$ and $q_{j_0}$ is an irreducible polynomial, therefore $q_{j_0}$ divides $P\circ\Psi- c_{j_0}.$ Now
		\begin{align*}
			X\cap Z_{j_0}\cap \overline{\mathfrak{D}}_2=\{z\in \overline{\mathfrak{D}}_2\setminus(L\cup \Gamma_{{\mathfrak{D}_{2}}}:q_{j_0}(z)=0, (h\circ\Psi)(z)- c_{j_0}=0)\}
		\end{align*}
		is finite. Hence 	
		\begin{align*}
			\widehat{\Gr_{\overline{P\circ \Psi}}(\Gamma_{\mathfrak{D}_{2}})}\cap \Gr_{\overline{P\circ \Psi}}\left(X\cap Z_{j_0}\cap \overline{\mathfrak{D}}_2\setminus (L\cup \Gamma_{{\mathfrak{D}_{2}}})\right)=\emptyset.
		\end{align*}
	\end{proof}
	
	\begin{remark}\label{R:Poly_Cnstnt_Variety}
		In view of \Cref{L:Hull_RegularSet} and \Cref{L:Z_jNot_InHull}, we obtain that if there exists a regular point $\alpha\in Z_{j}$ such that $\Gr_{\overline{(P\circ \Psi)}}(\alpha)\in \widehat{\Gr_{\overline{(P\circ \Psi)}}(\Gamma_{{\mathfrak{D}_{2}}})},$ then $P\circ\Psi=c_{j}$ and $h\circ \Psi=\overline{c}_{j}$ on $Z_{j}$ for some constant $c_{j}.$	
	\end{remark}

	\noindent We write
	\begin{align*}
		\mathfrak{J}_{0}:=\{j\in \{1,\cdots,l\}:\emptyset\not=\overline{\mathfrak{D}}_{2}\cap Z_{j}\setminus (L\cup \Gamma_{{\mathfrak{D}_{2}}})\subset X\}.
	\end{align*}
	And
	\begin{align}\label{E:On Zj}
		P\circ\Psi =c_{j} \text{ and }	h\circ\Psi=\overline{c_{j}},~j\in \mathfrak{J}_{0}.
	\end{align}

	\noindent From above lemmas we get that
	\begin{align*}
		\widehat{\Gr_{\overline{P\circ \Psi}}(\Gamma_{\mathfrak{D}_{2}})}\subset \Gr_{\overline{P\circ \Psi}}(L\cup \Gamma_{\mathfrak{D}_{2}})\cup\bigg(\cup_{j\in \mathfrak{J}_{0}}  \Gr_{\overline{P\circ \Psi}}(Z_{j}\cap \overline{\mathfrak{D}}_2\cap X)\bigg).
	\end{align*}
	
	\begin{lemma}\label{L:Q_j in Hull}
		If $j\in \mathfrak{J}_{0}$ and $Q_{j}$ as in (\ref{E:Defn_Qj}), then  $\Gr_{\overline{P\circ \Psi}}(\widehat{Q}_{j})  \subset\widehat{\Gr_{\overline{P\circ \Psi}}(\Gamma_{{\mathfrak{D}_{2}}})}.$ 
	\end{lemma}
	\begin{proof}
		Let $(a,\overline{(P\circ \Psi)}(a))\in \Gr_{\overline{P\circ \Psi}}(\widehat{Q}_{j})$ and let $H$ be any polynomial in $\cplx^3.$ Then
		\begin{align*}
			|H(a,\overline{(P\circ \Psi)(a)})|&\le \sup_{\Gr_{\overline{P\circ \Psi}}(\widehat{Q}_{j})}|H(z,\xi)|\\
			&=\sup_{\widehat{Q}_{j}}|H(z,\bar{c}_{j})|~\text {(by (\ref{E:On Zj}))}\\
			&=\sup_{{Q}_{j}}|H(z,\overline{(P\circ \Psi)(z)})|\\
			&\le \sup_{\Gamma_{\mathfrak{D}_{2}}}|H(z,\overline{(P\circ \Psi)(z)})|.
		\end{align*}
		Hence $(a,\overline{(P\circ \Psi)(a)})\in \widehat{\Gr_{\overline{P\circ \Psi}}(\Gamma_{{\mathfrak{D}_{2}}})}.$ Therefore, $\Gr_{\overline{P\circ \Psi}}(\widehat{Q}_{j})  \subset\widehat{\Gr_{\overline{P\circ \Psi}}(\Gamma_{{\mathfrak{D}_{2}}})}.$	
	\end{proof}
	
	\begin{lemma}\label{L:Outsde_Qj_Nt_InHull}
		If $j_0\in \mathfrak{J}_{0}$ and $Q_{j}$ as in (\ref{E:Defn_Qj}),	then  $\Gr_{\overline{P\circ \Psi}}\left(Z_{j_0}\cap \overline{\mathfrak{D}}_{2}\setminus(L\cup \Gamma_{{\mathfrak{D}_{2}}}\cup\widehat{Q}_{j_0})\right)  \cap\widehat{\Gr_{\overline{P\circ \Psi}}(\Gamma_{{\mathfrak{D}_{2}}})}=\emptyset.$ 
	\end{lemma}
	\begin{proof}
		We write $(P\circ \Psi)(z)-c_{j_0}=f_{1}(z)q^{m_{j_0}}_{j_0}(z),$ where $m_{j_0}$ denotes the maximal order of an irreducible factor $q_{j_0}.$ We define 
		$A=\{z\in \overline{\mathfrak{D}}_{2}:f_{1}(z)=0\}.$ Let $a_{0}\in \left(Z_{j_0}\cap \overline{\mathfrak{D}}_{2}\setminus(L\cup \Gamma_{{\mathfrak{D}_{2}}}\cup\widehat{Q}_{j_0})\right).$ First, we assume that $a_{0}\in Z_{j_0}\cap \overline{\mathfrak{D}}_2\setminus (L\cup A\cup \Gamma_{{\mathfrak{D}_{2}}}).$ We claim that $\Gr_{\overline{P\circ \Psi}}(a_{0})  \notin \widehat{\Gr_{\overline{P\circ \Psi}}(\Gamma_{{\mathfrak{D}_{2}}})}.$ To prove this claim, we put
		\begin{align*}
			f_2(z)=\frac{f_{1}(z)}{f_{1}(a_{0})}.
		\end{align*}
		Since $\widehat{Q}_{j_0},$ $\{a_{0}\}$ are disjoint polynomially convex set, there exists a neighborhood $U$ of $\widehat{Q}_{j_0}$ and a neighborhood $W$ of $A$, a polynomial $p_0$ such that 
		\begin{align*}
			p_{0}(a_0)=1,\quad |p_{0}(z)f_{2}(z)|<\frac{1}{2} ~~\text{ on } U, \text{ and }\\
			|p_{0}(z)f_{2}(z)|<\frac{1}{2} ~~\text{ on } W.
		\end{align*}
		We set $M=\sup _{\Gamma_{{\mathfrak{D}_{2}}}}|P\circ\Psi-c_{j_0}|$ and $K_{1}:=\{z\in \overline{\mathfrak{D}}_{2}:(P\circ\Psi)(z)-c_{j_0}=0\}.$ If we define 
		\begin{align*}
			g_{1}(z_1,z_2,\xi)=1-\frac{((P\circ\Psi)(z)-c_{j_0})(\xi-\bar{c}_{j_0})}{2M^2},
		\end{align*}
		then $g_1=1$ on $\Gr_{\overline{P\circ \Psi}}(K_1).$ Clearly, $|g_{1}|\le 1$ on $\Gamma_{{\mathfrak{D}_{2}}}.$ But $|g_1|\ne 1$ on $\Gamma_{{\mathfrak{D}_{2}}}\setminus (W\cup U).$ This implies $|g_1|<1$ on $\Gamma_{{\mathfrak{D}_{2}}}\setminus (W\cup U).$
		Therefore, there exists a natural number $k$ such that $|f_2(z)p_{0}(z)g_{1}(z,\xi)^{k}|<\frac{1}{2}$ on $\Gamma_{{\mathfrak{D}_{2}}}\setminus (W\cup U).$ If we take $g(z)=f_2(z)p_{0}(z)g_{1}(z,\xi)^{k},$ then $g((a_0,\overline{(P\circ\Psi)}(a_0)))=1.$ We claim that 
		\begin{align*}
			|g(z_1,z_2,\xi)|<\frac{1}{2} \text{ on } \Gr_{\overline{P\circ \Psi}}(\Gamma_{{\mathfrak{D}_{2}}}).	
		\end{align*}
		Assume that $\sup_{\Gr_{\overline{P\circ \Psi}}(\Gamma_{{\mathfrak{D}_{2}}})} |g(z_1,z_2,\xi)|=|g(s,\overline{P\circ \Psi}(s))|$ for some $s=(s_1,s_2)\in \Gamma_{{\mathfrak{D}_{2}}}.$
		
		\smallskip
		
		\noindent  {\bf Case I:} $s\in \Gamma_{{\mathfrak{D}_{2}}}\setminus (U\cup W),$ then there is nothing to prove.
		
		\smallskip
		
		\noindent  {\bf Case II:} $s\in U$ and $s\in \Gamma_{{\mathfrak{D}_{2}}}.$ Then $|g_{1}(s,\overline{P\circ \Psi}(s))|\le 1$ but $|p_{0}(s)f_{2}(s)|<\frac{1}{2}.$ This implies $|g(s,\overline{P\circ \Psi}(s))|<\frac{1}{2}.$
		
		\smallskip
		
		\noindent  {\bf Case III:} $s\in W$ and $s\in \Gamma_{{\mathfrak{D}_{2}}}.$ Then $|g_{1}(s,\overline{P\circ \Psi}(s))|\le 1$ but $|p_{0}(s)f_{2}(s)|<\frac{1}{2}.$ This implies $|g(s,\overline{P\circ \Psi}(s))|<\frac{1}{2}.$
		
		\noindent Combining all the three cases, we can conclude that $\Gr_{\overline{(P\circ \Psi)}}(a_0)\notin \widehat{\Gr_{\overline{(P\circ \Psi)}}(\Gamma_{{\mathfrak{D}_{2}}})}.$
		
		\noindent	Next, we assume that $a_0\in (Z_{j_0}\cap \overline{\mathfrak{D}}_{2}\cap A)\setminus (L\cup \Gamma_{{\mathfrak{D}_{2}}}).$ Since $(Z_{j_0}\cap \overline{\mathfrak{D}}_{2}\cap A)\setminus (L\cup \Gamma_{{\mathfrak{D}_{2}}})$ is finite and $\widehat{\Gr_{\overline{(P\circ \Psi)}}(\Gamma_{{\mathfrak{D}_{2}}})}$ does not containing any isolated point, therefore, in this case also $\Gr_{\overline{(P\circ \Psi)}}(a_0)\notin \widehat{\Gr_{\overline{(P\circ \Psi)}}(\Gamma_{{\mathfrak{D}_{2}}})}.$ 
		
	\end{proof}
	
	\noindent In view of \Cref{L:Outsde_Qj_Nt_InHull} and \Cref{L:Q_j in Hull}, we get that
	\begin{align*}
		\widehat{\Gr_{\overline{P\circ \Psi}}(\Gamma_{\mathfrak{D}_{2}})}=\Gr_{\overline{P\circ \Psi}}(\Gamma_{\mathfrak{D}_{2}})\cup\bigg(\cup_{j\in J}  \Gr_{\overline{P\circ \Psi}}(\widehat{Q_{j}})\bigg).
	\end{align*}

	\section{Examples}\label{sec-examples}
	
	\begin{example}\label{Exam:AppR_PolyHedrn}
		Let $p_{1}(z_1,z_2)=az_1+bz^l_2,~~p_{2}(z_1,z_2)=cz_1+dz^l_2$ $(l\in \mathbb{N},ad-bc\ne 0, |a|\ne |c|)$ and $\overline{\mathfrak{D}}_{2}:=\{|p_1(z)|\le 1,|p_{2}(z)|\le 1\}.$ Let $f_{j}\in \smoo(\Gamma_{\mathfrak{D}_{2}})$ for $j=1,\cdots,N, N\ge 1,$ and assume that each $f_{j}$ extends to a pluriharmonic function on $\mathfrak{D}_{2}.$ Then either
		\begin{align*}
			[z_1,z_2,f_1,\cdots,f_{N};\Gamma_{\mathfrak{D}_{2}}]=\smoo(\Gamma_{\mathfrak{D}_{2}}),
		\end{align*} 
		\noindent or there exists a non-trivial algebraic variety $V\subset \cplx^2$ with $V\cap \partial\mathfrak{D}_{2}\subset \Gamma_{\mathfrak{D}_{2}},$ and the pluriharmonic extensions of the $f_{j}$ are holomorphic on $V.$
		
		\smallskip
		
		\noindent {\bf Explanation}: We claim that
		\begin{enumerate}
			\item[(i)] $\Psi=(p_1,p_{2}):\cplx^2\to \cplx^2$ is a proper map;
			\item[(ii)] $\mathfrak{D}_{2}$ is complex non-degenerate;
			\item[(iii)] $\overline{\mathfrak{D}}_{2}$ is simply connected;
			\item[(iv)]$\{p_1(z)=e_1\},$ $\{p_2(z)=e_2\}$ is isomorphic to $\cplx$ for any $e_1,e_2\in \cplx.$
		\end{enumerate}
		
		\smallskip
		
		\noindent {\bf  Proof of (i)}: We consider the map $g(z_1,z_2)=(z_1,z^{l}_2)$ and $\phi(z_1,z_2)=(az_1+bz_2,cz_1+dz_2),$ then $\Psi=\phi\circ g.$ Since $g$ is proper map and $\phi$ is an automorphism of $\cplx^2,$ $\Psi$ is proper.\\
		
		\noindent{\bf  Proof of (ii)}: We compute
		\begin{align*}
			dp_1\wedge dp_2=l(ad-bc)z^{l-1}_2dz_1\wedge dz_2.
		\end{align*}
		Since $|a|\ne |c|$, $\Gamma_{\mathfrak{D}_{2}} \cap\{(z_1,z_2):z_2=0\}=\emptyset.$ Hence$~dp_1\wedge dp_2\ne 0$ on $\Gamma_{\mathfrak{D}_{2}}.$ Also it is easy to see that $dp_1\ne 0$ on $|p_1|=1$ and $|dp_2|\ne 0$ $|p_2|=1.$ Hence $\mathfrak{D}_{2}$ is complex non-degenerate.
		
		\smallskip
		
		\noindent {\bf  Proof of (iii)}:
		It is easy to see that $\overline{\mathfrak{D}}_{2}$  is path-connected. To see this, let $(a,b)\in \overline{\mathfrak{D}}_{2}.$ Then $\eta(t):=(t^{l}a,t b):[0,1]\to \overline{\mathfrak{D}}_{2}$ is a path between $(a,b)$ and the origin. Therefore, there is always a path between the origin and the arbitrary point of $\overline{\mathfrak{D}}_{2}.$ Hence $\overline{\mathfrak{D}}_{2}$ is path connected. 
		
		\smallskip 
		
		\noindent Now we show that $\overline{\mathfrak{D}}_{2}$ is simply-connected. For this, let $\gamma=(\gamma_1,\gamma_2):[0,1]\to\overline{\mathfrak{D}}_{2}$ be a loop in $\overline{\mathfrak{D}}_{2}$ based at $(0,0).$ We consider the function $H(t,x):[0,1]\times[0,1]\to \overline{\mathfrak{D}}_{2}$ by 
		\begin{align*}
			H(t,x)=\left(t^{l}\gamma_1(x),t\gamma_2(x)\right).
		\end{align*}
		Clearly, $H(t,x)\in \overline{\mathfrak{D}}_{2} ~~\forall t,x\in [0,1],~H(0,x) \equiv \gamma_{0}$ and $H(1,x)=\left(\gamma_1(x),\gamma_2(x)\right)=\gamma(x).$ Since $\gamma$ is continuous, $H$ is continuous. Hence $\gamma$ is homotopic to $\gamma_{0}.$ Since $\gamma$ is arbitrary, $\overline{\mathfrak{D}}_{2}$ is simply connected.\\
		
		\noindent {\bf Proof of (iv)}: Without loss of generality, assume that $|a|\ne 0\ne |c|.$ Then $p_{1}(z)=e_{1}$ and $p_{2}(z)=e_{2}$ are simply connected for any $e_1,e_2\in \cplx$. To prove this we use the following change of coordinate system $(z^{*}_1,z^{*}_{2})=(az_1+bz^{l}_2-e_1,z_{2})$ for the first one and $(z^{*}_1,z^{*}_{2})=(cz_1+dz^{l}_2-e_2,z_{2})$ for the second one. Then $p_{1}(z)=e_{1}$ and $p_{2}(z)=e_{2}$ becomes $\{(z^{*}_1,z^{*}_{2})\in \cplx^2:z^{*}_{1}=0\}$ and $\{(z^{*}_1,z^{*}_{2})\in \cplx^2:z^{*}_{1}=0\}$ respectively.	
		
		\smallskip
		
		\noindent Now by using \Cref{T:PPolyhrn_DistingBdry}, we can conclude that either
		\begin{align*}
			[z_1,z_2,f_1,\cdots,f_{N};\Gamma_{\mathfrak{D}_{2}}]=\smoo(\Gamma_{\mathfrak{D}_{2}}),
		\end{align*} 
		
		\noindent or there exists a non-trivial algebraic variety $V\subset \cplx^2$ with $V\cap \partial\mathfrak{D}_{2}\subset \Gamma_{\mathfrak{D}_{2}},$ and the pluriharmonic extensions of the $f_{j}$ are holomorphic on $V.$	
	\end{example}

	\begin{example}
		Let $p_{1}(z_1,z_2)=z_1+z_2,~~p_{2}(z_1,z_2)=z_1z_2,~~P(z_1,z_2)=z_1+z_{2}$ and $\Psi(z_1,z_2)=(p_{1}(z_1,z_2),p_{2}(z_1,z_2)).$ Then the graph of $\overline{P\circ \Psi}$ over $\Gamma_{{\mathfrak{D}_{2}}}$ is not polynomially convex. Hence polynomial approximation does not holds.
	\end{example}	
	
	\noindent {\bf Explanation}: Following the notations in \Cref{T:Analogus_Jimbo}, let $\Psi(z)=(p_1(z),p_{2}(z)),$ $h(z)=\frac{1}{z_1}+\frac{1}{z_2}.$ Therefore,
	\begin{align*}
		X&=\left\{z\in \overline{\mathfrak{D}_{2}}\setminus (L\cup \Gamma_{\mathfrak{D}_{2}}): \overline{(P\circ \Psi)(z)}=(h\circ\Psi)(z)\right\}\\
		&=\left\{z\in \overline{\mathfrak{D}_{2}}\setminus (L\cup \Gamma_{\mathfrak{D}_{2}}): \overline{z_1+z_2+z_1z_2}=\frac{1}{z_1+z_2}+\frac{1}{z_1z_2}\right\}.
	\end{align*}	
	\noindent and	
	\begin{align*}
		\triangle({z})&=
		\begin{vmatrix} 
			\frac{\partial P(\Psi)}{\partial{z_1}}& \frac{\partial P(\Psi)}{\partial z_2}\\[1.5ex]
			\frac{\partial h(\Psi)}{\partial{{z_1}}}& \frac{\partial h(\Psi)}{\partial{z_2}}\\[1.5ex]
		\end{vmatrix}
		\begin{vmatrix} 
			\frac{\partial p_1}{\partial{{z_1}}}& \frac{\partial p_1}{\partial{z_2}}\\[1.5ex]
			\frac{\partial p_2}{\partial{{z_1}}}& \frac{\partial p_2}{\partial{z_2}}\\[1.5ex]
		\end{vmatrix}_{z}
		=\begin{vmatrix}
			1&1\\
			\frac{-1}{p^{2}_{1}(z)} & \frac{-1}{p^{2}_{2}(z)}
		\end{vmatrix}\begin{vmatrix}
			1 & 1\\
			z_2 & z_1
		\end{vmatrix}\\
		&=\frac{1}{p^2_{1}p^2_{2}}(p_{1}+p_{2})(p_{1}-p_{2})(z_2-z_1).
	\end{align*}
	\noindent We define $q_1:=p_{1}-p_{2}=z_1+z_2-z_1z_2,~~q_{2}:=p_{1}+p_{2}=z_1+z_2+z_1z_2,~~q_3=z_2-z_1$ and $Z_{j}=\{q_{j}=0\}.$ Therefore,

	\noindent Clearly, $Z_{j}\cap \overline{\mathfrak{D}}_2\setminus (\Gamma_{{\mathfrak{D}_{2}}}\cup L)\subset X$ only for ${j}=2.$ It is easy to see that $\widehat{Q}_{2}=\widehat{Z_{2}\cap \Gamma_{{\mathfrak{D}_{2}}}}=\{(z_1,z_2)\in\overline{\mathfrak{D}}_{2}: z_1+z_2+z_1z_2=0 \}.$ Indeed, we write $\mathcal{K}=\{(z_1,z_2)\in \cplx^2: z_1+z_2=0,|z_1|=1,|z_2|=1\}.$ Clearly, $\Psi^{-1}(\mathcal{K})=Q_{2}=Z_{2}\cap \Gamma_{{\mathfrak{D}_{2}}}.$ We now claim that $$\Psi^{-1}(\widehat{\mathcal{K}})=\widehat{Q_{2}}.$$
	
	\noindent Note that
	\begin{align*}
		\Psi(Q_{2})\subset\Psi(\widehat{Q_{2}})	\subset \Psi(Z_2\cap \overline{\mathfrak{D}}_2)=\widehat{\mathcal{K}}.
	\end{align*}
	This implies
	\begin{align*}
		\widehat{\mathcal{K}}=\widehat{	\Psi(Q_{2})}\subset\widehat{\Psi(\widehat{Q_{2}})}	\subset \Psi(Z_2\cap \overline{\mathfrak{D}}_2)=\widehat{\mathcal{K}}.
	\end{align*}
	Hence $\widehat{\mathcal{K}}=\widehat{	\Psi(Q_{2})}=\widehat{\Psi(\widehat{Q_{2}})}.$ Since $\Psi^{-1}(\Psi(\widehat{Q_2}))=\widehat{Q_2}$ is polynomially convex and $\Psi$ is proper map, therefore $\Psi(\widehat{Q_2})$ is also polynomially convex. Thus $\widehat{	\Psi(Q_{2})}=\widehat{\Psi(\widehat{Q_{2}})}=\Psi(\widehat{Q_{2}}).$ This implies $\widehat{Q_{2}}=\Psi^{-1}(\widehat{\mathcal{K}}),$ i.e. $\widehat{Q_{2}}=\overline{\mathfrak{D}}_2\cap Z_{2}.$

	\noindent Therefore, by \Cref{T:Analogus_Jimbo}, we get that	
	\begin{align*}
		\widehat{\Gr_{\overline{P\circ \Psi}}(\Gamma_{\mathfrak{D}_{2}})}=\Gr_{\overline{P\circ \Psi}}(\Gamma_{\mathfrak{D}_{2}})\cup \{(z_1,z_2,0):(z_1,z_2)\in \overline{\mathfrak{D}}_2, z_1+z_2+z_1z_2=0 \}.
	\end{align*}

	\begin{example}
		$p_{1}(z_1,z_2)=z_1+z_2,~~p_{2}(z_1,z_2)=z_1z_2,~~P(z_1,z_2)=z_1+2z_{2}$ and $\Psi(z_1,z_2)=(p_{1}(z_1,z_2),p_{2}(z_1,z_2)).$ Then the graph of $\overline{P\circ \Psi}$ over $\Gamma_{{\mathfrak{D}_{2}}}$ is polynomially convex. Hence polynomial approximation holds.
	\end{example}
	
	\smallskip
	
	\noindent {\bf Explanation}:
	Following the notation in \Cref{T:Analogus_Jimbo}, let $\Psi(z)=(p_1(z),p_{2}(z)),$ $h(z)=\frac{1}{z_1}+\frac{2}{z_2}.$ Therefore,
	\begin{align*}
		X&=\left\{z\in \overline{\mathfrak{D}_{2}}\setminus (L\cup \Gamma_{\overline{\mathfrak{D}_{2}}}): \overline{(P\circ \Psi)(z)}=(h\circ\Psi)(z)\right\}\\
		&=\left\{z\in \overline{\mathfrak{D}_{2}}\setminus (L\cup \Gamma_{\mathfrak{D}_{2}}): \overline{z_1+z_2+2z_1z_2}=\frac{1}{z_1+z_2}+\frac{2}{z_1z_2}\right\},
	\end{align*}	
	\noindent and	
	\begin{align*}
		\triangle({z})&=
		\begin{vmatrix} 
			\frac{\partial P(\Psi)}{\partial{z_1}}& \frac{\partial P(\Psi)}{\partial z_2}\\[1.5ex]
			\frac{\partial h(\Psi)}{\partial{{z_1}}}& \frac{\partial h(\Psi)}{\partial{z_2}}\\[1.5ex]
		\end{vmatrix}
		\begin{vmatrix} 
			\frac{\partial p_1}{\partial{{z_1}}}& \frac{\partial p_1}{\partial{z_2}}\\[1.5ex]
			\frac{\partial p_2}{\partial{{z_1}}}& \frac{\partial p_2}{\partial{z_2}}\\[1.5ex]
		\end{vmatrix}_{z}
		=\begin{vmatrix}
			1& 2\\
			\frac{-1}{p^{2}_{1}(z)} & \frac{-2}{p^{2}_{2}(z)}
		\end{vmatrix}
		\begin{vmatrix}
			1 & 1\\
			z_2 & z_1
		\end{vmatrix}\\
		&=\frac{2}{p^2_{1}p^2_{2}}(p_{1}+p_{2})(p_{1}-p_{2})(z_2-z_1).
	\end{align*}
	\noindent	We define $q_1:=p_{1}-p_{2}=z_1+z_2-z_1z_2,~~q_{2}:=p_{1}+p_{2}=z_1+z_2+z_1z_2,~~q_3=z_2-z_1$ and $Z_{j}=\{q_{j}=0\}.$ It is easy see that $Z_{j}\cap \overline{\mathfrak{D}}_2\setminus (\Gamma_{{\mathfrak{D}_{2}}}\cup L)\nsubseteq X ~\forall {j}\in\{1,2,3\}.$ Hence, 	
	\begin{align*}
		\widehat{\Gr_{\overline{P\circ \Psi}}(\Gamma_{\mathfrak{D}_{2}})}=\Gr_{\overline{P\circ \Psi}}(\Gamma_{\overline{\mathfrak{D}_{2}}}).
	\end{align*}
	Therefore, by using \Cref{T: Approx_Cont_Func}, we get that	 $[z_1,z_2,\overline{(P\circ \Psi)}(z);\Gamma_{\mathfrak{D}_2}]=\smoo(\Gamma_{\mathfrak{D}_2}).$

	\smallskip
	
	\begin{example}
		Let $p_{1}(z_1,z_2)=2z_1+z^2_2,~~p_{2}(z_1,z_2)=z_1-z^2_2,~~P(z_1,z_2)=z_1-z_{2}$ and $\Psi(z_1,z_2)=(p_{1}(z_1,z_2),p_{2}(z_1,z_2)).$ Then the graph of $\overline{P\circ \Psi}$ over $\Gamma_{{\mathfrak{D}_{2}}}$ is not polynomially convex. Hence polynomial approximation does not holds.
	\end{example}

	\noindent {\bf Explanation}:
	Following the notation in \Cref{T:Analogus_Jimbo}, let $\Psi(z)=(p_1(z),p_{2}(z)),$ $h(z)=\frac{1}{z_1}-\frac{1}{z_2}.$ 
	Therefore,
	\begin{align*}
		X&=\left\{z\in \overline{\mathfrak{D}_{2}}\setminus (L\cup \Gamma_{\mathfrak{D}_{2}}): \overline{(P\circ \Psi)(z)}=(h\circ\Psi)(z)\right\}\\
		&=\left\{z\in \overline{\mathfrak{D}_{2}}\setminus (L\cup \Gamma_{\mathfrak{D}_{2}}): \overline{z_1+2z^2_2}=\frac{1}{2z_1+z^2_2}+\frac{1}{z_1-z^2_2}\right\}.
	\end{align*}
	\noindent and 			
	\begin{align*}
		\triangle({z})&=
		\begin{vmatrix} 
			\frac{\partial P(\Psi)}{\partial{z_1}}& \frac{\partial P(\Psi)}{\partial z_2}\\[1.5ex]
			\frac{\partial h(\Psi)}{\partial{{z_1}}}& \frac{\partial h(\Psi)}{\partial{z_2}}\\[1.5ex]
		\end{vmatrix}
		\begin{vmatrix} 
			\frac{\partial p_1}{\partial{{z_1}}}& \frac{\partial p_1}{\partial{z_2}}\\[1.5ex]
			\frac{\partial p_2}{\partial{{z_1}}}& \frac{\partial p_2}{\partial{z_2}}\\[1.5ex]
		\end{vmatrix}_{z}
		=\begin{vmatrix}
			1 & -1\\
			\frac{-1}{p^{2}_{1}(z)} & \frac{1}{p^{2}_{2}(z)}
		\end{vmatrix}\begin{vmatrix}
			2 & 2z_2\\
			1 & -2z_2
		\end{vmatrix}\\
		&=\frac{1}{p^2_{1}p^2_{2}}(p_{1}+p_{2})(p_{1}-p_{2})(-6z_2).
	\end{align*}
	\noindent We define $q_1:=p_{1}-p_{2}=z_1+2z^2_2,~~q_{2}:=p_{1}+p_{2}=3z_1,~~q_3=-6z_2$ and $Z_{j}=\{q_{j}=0\}.$
	
	\noindent Clearly, $Q_{3}=Z_{3}\cap \Gamma_{{\mathfrak{D}_{2}}}=\emptyset$ and it is easy to show that if  $z\in Z_{2}\cap X,$ then $z\in \Gamma_{{\mathfrak{D}_{2}}},$ this is not possible because $X\cap \Gamma_{{\mathfrak{D}_{2}}}=\emptyset.$
	Now its remaining to verify for the set $Q_{1}=Z_{1}\cap \Gamma_{{\mathfrak{D}_{2}}}.$ It is easy to see that 
	$Z_{1}\cap \overline{\mathfrak{D}}_2\setminus (\Gamma_{{\mathfrak{D}_{2}}}\cup L)\subset X$ and $\widehat{Q}_{1}=\widehat{Z_{1}\cap \Gamma_{{\mathfrak{D}_{2}}}}=\{(z_1,z_2)\in\overline{\mathfrak{D}}_{2}: z_1+2z^2_2=0 \}.$ 
	Therefore, by \Cref{T:Analogus_Jimbo}, we get that	
	\begin{align*}
		\widehat{\Gr_{\overline{P\circ \Psi}}(\Gamma_{\mathfrak{D}_{2}})}=\Gr_{\overline{P\circ \Psi}}(\Gamma_{\mathfrak{D}_{2}})\cup \{(z_1,z_2,0):(z_1,z_2)\in \overline{\mathfrak{D}}_2, z_1+2z^2_2=0 \}.
	\end{align*}
	
	\smallskip
	
	\begin{example}
		Let $p_{1}(z_1,z_2)=2z_1+z^2_2,~~p_{2}(z_1,z_2)=z_1-z^2_2,~~P(z_1,z_2)=z_1+2z_{2}$ and $\Psi(z_1,z_2)=(p_{1}(z_1,z_2),p_{2}(z_1,z_2)).$ Then the graph of $\overline{P\circ \Psi}$ over $\Gamma_{{\mathfrak{D}_{2}}}$ is polynomially convex. 
	\end{example}	
	
	\smallskip
	
	\noindent {\bf Explanation}:
	Following the notations in \Cref{T:Analogus_Jimbo}, let $\Psi(z)=(p_1(z),p_{2}(z)),$ $h(z)=\frac{1}{z_1}+\frac{2}{z_2}.$
	Therefore,	
	\begin{align*}
		X&=\left\{z\in \overline{\mathfrak{D}_{2}}\setminus (L\cup \Gamma_{\mathfrak{D}_{2}}): \overline{(P\circ \Psi)(z)}=(h\circ\Psi)(z)\right\}\\
		&=\left\{z\in \overline{\mathfrak{D}_{2}}\setminus (L\cup \Gamma_{\mathfrak{D}_{2}}): \overline{4z_1-z^2_2}=\frac{1}{2z_1+z^2_2}+\frac{2}{z_1-z^2_2}\right\}.
	\end{align*} 	
	\noindent and	
	\begin{align*}
		\triangle({z})&=
		\begin{vmatrix} 
			\frac{\partial P(\Psi)}{\partial{z_1}}& \frac{\partial P(\Psi)}{\partial z_2}\\[1.5ex]
			\frac{\partial h(\Psi)}{\partial{{z_1}}}& \frac{\partial h(\Psi)}{\partial{z_2}}\\[1.5ex]
		\end{vmatrix}
		\begin{vmatrix} 
			\frac{\partial p_1}{\partial{{z_1}}}& \frac{\partial p_1}{\partial{z_2}}\\[1.5ex]
			\frac{\partial p_2}{\partial{{z_1}}}& \frac{\partial p_2}{\partial{z_2}}\\[1.5ex]
		\end{vmatrix}_{z}
		=\begin{vmatrix}
			1 & 2\\
			\frac{-1}{p^{2}_{1}(z)} & \frac{-2}{p^{2}_{2}(z)}
		\end{vmatrix}\begin{vmatrix}
			2 & 2z_2\\
			1 & -2z_2
		\end{vmatrix}\\
		&=\frac{1}{p^2_{1}p^2_{2}}(p_{1}+p_{2})(p_{1}-p_{2})(6z_2).
	\end{align*}
	\noindent We define $q_1:=p_{1}-p_{2}=z_1+2z^2_2,~~q_{2}:=p_{1}+p_{2}=3z_1,~~q_3=z_2$ and $Z_{j}=\{q_{j}=0\}.$ Clearly, $Q_{3}=Z_{3}\cap \Gamma_{{\mathfrak{D}_{2}}}=\emptyset$ and it is easy to show that if  $z\in Z_{j}\cap X,$ then $z\in \Gamma_{{\mathfrak{D}_{2}}},$ for $j=1,2.$ This is not possible because $X\cap \Gamma_{{\mathfrak{D}_{2}}}=\emptyset.$ Therefore, by \Cref{T:Analogus_Jimbo}, we get that	
	\begin{align*}
		\widehat{\Gr_{\overline{P\circ \Psi}}(\Gamma_{\mathfrak{D}_{2}})}=\Gr_{\overline{P\circ \Psi}}(\Gamma_{\mathfrak{D}_{2}}).
	\end{align*}
	Therefore, by using \Cref{T: Approx_Cont_Func}, we get $[z_1,z_2,\overline{(P\circ \Psi)}(z);\Gamma_{\mathfrak{D}_2}]=\smoo(\Gamma_{\mathfrak{D}_2}).$
	
	\medskip

		\noindent{\bf Acknowledgments.} The first named author was partially supported by a MATRICS Research Grant (MTR/2017/000974) of SERB, Dept. of Science and Technology, Govt. of India, for the beginning of this work and is supported partially by a Core Research Grant (CRG/2022/003560) of ANRF (formerly SERB), Govt. of India,  for the later part of the work. The second named author acknowledges partial support from an INSPIRE Fellowship (IF 160487) awarded by the DST, Govt. of India, for the early version of this work, and an IoE-IISc Postdoctoral Fellowship (R(HR)(IOE-PDF)(MA)(GMM)-114) for the latest version of this work. The second named author would also like to thank Amar Deep Sarkar for the fruitful discussions.


\end{document}